%% file: main_arxiv.tex
\documentclass[twoside,preprint]{article}

\usepackage{aistats2027_arxiv}

\input{preamble_arxiv}
\usepackage{xr}
\newcommand{\shift}{0}
\newcommand{\aperture}{-1}

\begin{document}

\twocolumn[

\aistatstitle{\CurrentTitle}

\aistatsauthor{Jose Cribeiro-Ramallo \And Florian Kalinke \And Zoltán Szabó}

\aistatsaddress{Karlsruhe Institute of Technology \And Karlsruhe Institute of Technology \And London School of Economics} ]

\begin{abstract}%
Kernel methods are among the most powerful tools in machine learning and statistics, with a large number of successful applications. Their immense success stems from the flexible function class  associated to each kernel---its reproducing kernel Hilbert space (RKHS)---which facilitates statistical analysis, as well as from their computational tractability and applicability to many domains.
Multiple notions (such as characteristic, $L_p$-universal, and integrally strictly positive definite) capture the expressivity of kernels and their RKHSs and play a key role in understanding the statistical properties of kernel methods; these concepts and their relations are well-understood for bounded kernels.
Even though unbounded kernels have received significant attention over the past decade (for instance, in the construction of  kernel-based discrepancy and dependence measures such as the maximum mean discrepancy, the Hilbert-Schmidt independence criterion, and the kernel Stein discrepancy), surprisingly little is known about the relations of these notions in the unbounded case. In the present paper we tackle this severe bottleneck, establishing their relations under mild assumptions.
\end{abstract}

\section{INTRODUCTION}
Kernel methods \citep{steinwart08support,saitoh16theory,paulsen16introduction} are among the most flexible and powerful tools in machine learning and statistics, demonstrating superior performance in a wide variety of applications. These range from classification and regression with support vector machines (SVMs; \citealt{vapnik97support}) and non-linear principal component analysis (PCA; \citealt{scholkopf98nonlinear}), to more recent advancements in the representation of probability measures \citep{berlinet04reproducing} leading to state-of-the-art advancements in non-parametric statistics \citep{schrab25discrepancies}.
The key idea of kernel methods is to map the data into a (possibly infinite-dimensional) feature space, a reproducing kernel Hilbert space (RKHS; \citealt{aronszajn50theory}); their RKHSs are the source of their flexibility and power, rendering the study of their expressivity fundamentally important.

There exist different notions to capture the expressivity of an RKHS, usually involving properties of its associated kernel function (since the RKHS-kernel association is one-to-one). For example, an RKHS that can represent probability measures without any loss of information is said to have a \emph{characteristic} kernel \citep{fukumizu08kernel2,sriperumbudur10hilbert}, while one capable of approximating functions on different function classes (classically the space of continuous bounded functions, or more broadly the $p$-integrable functions; \citealt{steinwart01influence,carmeli10vector}) is said to have a \emph{universal} kernel.
Further, when the kernel is bounded, these properties are known to be related to (i) each other, (ii)  \emph{integrally strictly positive definite} (i.s.p.d.) kernels \citep{sriperumbudur11universality},  and (iii) to semimetric spaces of \emph{strong negative type} (s.n.t.; \citealt{lyons13distance,sejdinovic13equivalence}). These notions often help to determine the validity of kernel-based information theoretical measures (ITMs), meant as kernel discrepancy and dependence measures.

Indeed, most kernel ITMs rely on the kernel mean embedding \citep{berlinet04reproducing,smola07hilbert}, a lossless (under a characteristic kernel) map of Borel probability measures into an RKHS. The RKHS distance between two mean embeddings results in the so-called maximum mean discrepancy (MMD; \citealt{gretton12kernel}), known to be equivalent \citep{sejdinovic13equivalence} to energy distance \citep{baringhaus04new,szekely04testing,szekely05new}---also known as $N$-distance \citep{zinger92characterization,klebanov05ndistance}. When the kernel is bounded, MMD is a metric on the space of Borel probability measures provided that the kernel is characteristic, and it further metrizes the weak convergence of probability measures when the kernel is universal to an appropriate function class \citep[Theorem 9]{simon-gabriel23metrizing}.
If  MMD is applied to a product space---with a product kernel---one gets the Hilbert-Schmidt independence criterion (HSIC). HSIC was originally designed for $M=2$ components \citep{gretton05measuring,gretton05kernel}, and later extended to $M\geq2$ \citep{quadrianto09kernelized2,sejdinovic13kernel2,pfister18kernel}. When the product kernel is bounded, HSIC is a valid independence measure for $M=2$ if the kernel components are all characteristic \citep{lyons13distance}; requiring universality suffices for $M>2$ \citep{szabo18characteristic2}.

Recently, new kernel-based ITMs requiring unbounded kernels have grown in popularity. In particular, when applying MMD to a target and a sampling distribution, and selecting the kernel such that the mean embedding of the target vanishes (via a Stein operator; \citealt{stein72bound,chen21stein,anastasiou23stein}), one obtains the kernel Stein discrepancy (KSD; \citealt{chwialkowski16kernel,liu16kernelized}); the kernel employed in the KSD (known as the Stein kernel) is practically never bounded (\citealt[Example~1]{kalinke25nystromksd}; \citealt[Remark~4]{hagrass26stein}). To guarantee that KSD is a valid goodness-of-fit measure, the mean embedding of the target $\P_0$ has to be different than the mean embedding of any other probability measure (a weaker condition than regular characteristic-ness known as \emph{$\P_0$-separating} or \emph{characteristic with regard to $\P_0$}; \citealt{barp24targeted,cribeiro26minimax2}); sufficient conditions for this specific property are known on $\R^d$ \citep{chwialkowski16kernel,barp24targeted}. Recently, \citet{modeste24characterization} %
studied sufficient conditions for a specific class of unbounded kernels on $\R^d$ to be characteristic; yet,
in general, the relations of the characteristic, universality, i.s.p.d., and s.n.t.\ properties associated to unbounded kernels on topological domains are open. Our primary focus is to fill in this severe gap.

In particular, we make the following \tb{contributions}.
\begin{enumerate}[label=(\roman*)]
    \item Building on existing notions of expressivity of RKHSs (and filling them in where required: $\mF$-s.n.t.\ and $L_p$-universality of unbounded kernels),
    we establish their relations on general domains and for unbounded kernels, summarized in Table~\ref{tab:relations} and Figure~\ref{fig:summary}.

    \item Our results imply, for example, that (a)
    the i.s.p.d.\ and the characteristic properties are equivalent for sub-groups of finite signed measures [Remark~\ref{rem:4}\ref{remark4:(a)}], (b)
    no Stein kernel can be $L_p$-universal for any $p\in[1,\infty)$ [Remark~\ref{rem:9}\ref{rem:9:d}] and (c) the exponential kernel is $L_p$-universal for all $p\in[1,\infty)$ in $\R^d$ [Remark~\ref{rem:9}\ref{rem:9:e}].

    \item \citet[Corollary 4]{carmeli10vector} showed the equivalence among $L_p$-universality for all values of $p\in[1,\infty)$ for $c_0$-kernels on locally compact second-countable spaces. A surprising phenomenon revealed by our results is that this behavior persists even when the kernel is unbounded (under mild conditions).
\end{enumerate}

For performing the analysis, we introduce and
study a union of sets of finite signed measures that are absolutely continuous w.r.t.\ kernel-dependent sets of $p$-integrable measures, in turn satisfying $1-1/p$-integrability conditions.

The paper is structured as follows. Section~\ref{sec:notations} introduces the notations and definitions used through the paper, followed by our results in  Section~\ref{sec:results}, with proofs in the supplement (we locate each proof in the respective statement).

\begin{table*}
    \caption{Summary of our main results.}
    \label{tab:results-in-the-main}
    \begin{center}
    \begin{tabular}{lll}
    \toprule
    Statement & Content & Constraint \\
    \midrule
    Theorem~\ref{th:snt_iff_ispd} & $(\X,\rho_K)$ has $\mF$-s.n.t.\ $\Leftrightarrow$\ $K$ is $\mF$-i.s.p.d.\ & $\mF\subseteq [\mPb^2(K;\X)]^0$\\
    Theorem~\ref{th:ispd0_iff_char} & $K$ is $\mF$-characteristic $\Rightarrow$ $K$ is $\mF$-i.s.p.d.\ & $0 \in \mF\subseteq \mPb(K;\X)$\\
    Theorem~\ref{th:ispd0_iff_char} & $K$ is $\mF$-characteristic $\Leftarrow$ $K$ is $\mF$-i.s.p.d.\ & $(\mF,+)\subseteq \mPb(K;\X)$ is a group\\
    Theorem~\ref{th:char_p1_iff_pb0} & $K$ is $\mP^p(K;\X)$-characteristic $\Leftrightarrow$ $K$ is $[\mPb^p(K;\X)]^0$-characteristic & $p\in [1,\infty)$\\
    Theorem~\ref{th:univ} & $K$ is $L_p$-universal $\Leftrightarrow$ $K$ is $\A^p(K;\X)$-i.s.p.d.\ & $p\in [1,\infty)$, $\H_K$: separable\\
    Theorem~\ref{th:Ap=Pb} & $\A^p(K;\X) = \mPb(K;\X)$ & $p\in[1,\infty)$\\
    \bottomrule
    \end{tabular}
    \end{center}
\end{table*}

\section{NOTATIONS}\label{sec:notations}

In this section we introduce the following notations: $\N$, $\Np$, $[d]$, $\N_{\ge d}$, $\times_{i=1}^d\X_i$, $\X^d$, $p'$, $\fip{\cdot,\cdot}{}$, $\norm{\cdot}{2}$, $\mB_\X$, $\overline{A}$, $\mC(\X)$, $\mL_p(\X,\nu)$, $\norm{\cdot}{L_p}$, $L_p(\X,\nu)$, $\fip{\cdot,\cdot}{E^*,E}$, $S^*$,  $\M_b(\X)$, $\M_1^+(\X)$, $[\mF]^0$, $\delta_{x}$, $\tfrac{\d\P}{\d\Q}$, $\F^+$, $\F^-$, $|\F|$, $\E_\P$, $K$, $K(\cdot,x)$, $\H_K$, $\norm{\cdot}{\H_K}$, $\mPb^p(K;\X)$, $\mP^p(K;\X)$, $\mPb(K;\X)$, $\mP(K;\X)$, $\mu_K(\cdot)$, $\id$, $S_K$, $\Abs_{p'}(\P)$, $\A^p(K;\X)$, $\A(K;\X)$, $\rho$, $\mPb^p(\rho;\X)$, $\mP^p(\rho;\X)$, $\mPb(\rho;\X)$, $\mP(\rho;\X)$, $\rho_K$, $K_\rho$, $\phi$, $\beta_\phi$.
Along the way, we also define the notions of (in order of definition) $\mF$-characteristic, $\mF$-i.s.p.d., $L_p$-universality of unbounded kernels, and $\mF$-s.n.t.; their specialization to classical notions is summarized in Figure~\ref{tab:relations} for the convenience of the reader.

\begin{figure}[t]
\centering
    \begin{equation}
    \xymatrix@C=6.8pc @R=0.7pc{
    {\begin{array}{@{}c@{}}
     \text{$\mF$-char.}
   \end{array}} \ar@{-->}[r]^{\scriptscriptstyle \mF=\M_1^+(\X),\ K\text{ bounded}}_{\scriptscriptstyle\text{\citep{fukumizu08kernel2}}} & \text{characteristic}\\
    \text{$\mF$-i.s.p.d.}  \ar@{-->}@{-->}[r]^{\scriptscriptstyle \mF=\M_b(\X),\ K\text{ bounded}}_{\scriptscriptstyle\text{\citep{sriperumbudur10hilbert}}} & \text{i.s.p.d.}\\
    {\begin{array}{@{}c@{}}
     \text{$L_p$-universal}
   \end{array}} \ar@{-->}[r]^{\scriptscriptstyle \mP^p(K;\X) = \M_1^+(\X)}_{\scriptscriptstyle\text{\citep{carmeli10vector}}} & {\begin{array}{@{}c@{}}
     \text{$L_p$-universal}\\
     \text{(bounded $K$)}
   \end{array}}\\
   {\begin{array}{@{}c@{}}
     \text{$\mF$-s.n.t.\ }
   \end{array}}\ar@{-->}[r]^{\scriptscriptstyle \mF = [\mPb^2(\rho;\X)]^0}_{\scriptscriptstyle\text{\citep{sejdinovic13equivalence}}} & \text{ s.n.t.}\\
    {\begin{array}{@{}c@{}}
     \text{($\mF,\F$)-char.}
   \end{array}} \ar@{-->}[r]^{\scriptscriptstyle \mF=\mP(K;\X),\ \F=\P_0}_{\scriptscriptstyle\text{\citep{cribeiro26minimax2}}} & {\begin{array}{@{}c@{}}
     \text{characteristic}\\
     \text{w.r.t.\ } \P_0
   \end{array}}
        }
    \end{equation}
\caption{Relations between the definitions used with classical notions in the literature. The items are ordered by the year of the introduction of the concepts on the r.h.s.}
\label{tab:relations}
\end{figure}

\tb{General:} Let $\N=\{0,1,2,\dots\}$ denote the set of  natural numbers. For a positive integer $d\in\Np=\{1,2,\dots\}$, let $[d]=\{1,2,\dots,d\}$, and $\N_{\ge d} = \{d,d+1,d+2,\dots\}$. For a sequence of sets $(\X_i)_{i=1}^d$, we denote their Cartesian product by $\times_{i=1}^d\X_i$; when $\X_i=\X$ for all $i\in[d]$, we write $\X^d$. For a $p\in [1,\infty]$, its conjugate exponent  $p'\in [1,\infty]$ is defined as $\frac{1}{p}+\frac{1}{p'}=1$, with the convention $\frac{1}{\infty}\coloneq 0$. The inner product of vectors $\b a,\b b\in \R^d$ is $\fip{\b a, \b b}{}=\sum_{i\in [d]}a_i b_i$; the Euclidean norm of $\b a\in \R^d$ is $\norm{\b a}{2} = \sqrt{\fip{\b a, \b a}{}}$. Throughout the paper, unless stated otherwise, let $(\X,\tau_{\X})$ be a topological space endowed with the Borel $\sigma$-algebra $\mB_\X\coloneq\mB(\X,\tau_{\X})$. For a set $A\subseteq\X$ its closure $\overline A$ is defined as the smallest (meant w.r.t.\ the containing relation) closed subset of $\X$ containing $A$. A set $A\subseteq\X$ is called dense in $\X$ if $\overline A = \X$.
The space of continuous real-valued functions on $\X$ is denoted by $\mC(\X)$.
For a measure space $(\X,\mB_\X,\nu)$ with a finite non-negative measure $\nu$ and $p\in[1,\infty)$, $\mL_p(\X,\nu) \coloneq \mL_p(\X,\mB_\X,\nu)$ is the space of $p$-integrable real-valued functions on $\X$; in other words, those measurable functions $f:\X\to\R$ for which
\begin{equation}
\norm{f}{L_p} \coloneq \norm{f}{L_p(\X,\mB_\X,\nu)} \coloneq \left[\int_\X |f(x)|^{p} \d \nu(x)\right]^{1/p} <\infty.
\end{equation}
We say that $f,f'\in \mL_p(\X,\nu)$ are equivalent, written as $f\sim f'$,
if  $\norm{f-f'}{L_p}=0$, in other words, $f\sim f'$ if and only if (iff) $f(x)=f'(x)$ for $\nu$-almost all $x\in\X$ ($\nu$-almost everywhere; $\nu$-a.e.). We define the set
\begin{equation}
    L_p(\X,\nu) \coloneq L_p(\X,\mB_\X,\nu)\coloneq \left\{[f]_\sim:~f\in\mL_p(\X,\nu)\right\},
\end{equation}
where $[f]_\sim=\{f'\in\mL_p(\X,\nu):~f'\sim f\}$ is the equivalence class of $f$.
Further, for a finite non-negative measure $\nu$ the space of ($\nu$-)essentially bounded functions $L_{\infty}(\X,\nu) \coloneq L_{\infty}(\X,\mB_\X,\nu)$ consists of those measurable functions  $f:\X \to \R$  for which
\begin{align}
\norm{f}{L_{\infty}} &\coloneq \norm{f}{L_{\infty}(\X,\mB_\X,\nu)} \\
&\coloneq \inf\left\{a>0:~\nu\big(\{x\in\X: |f(x)|>a\}\big)=0 \right\} \notag
\end{align}
is finite.
For $p\in [1,\infty]$, $L_p(\X,\nu)$ enriched with the norm $\norm{[f]_\sim}{L_p}\coloneq \norm{f}{L_p}$ is a Banach space [\citealt[(A.33)]{steinwart08support}].
Consider two Banach spaces $E$ and $F$. A map $S:E\to F$ for which $S(\alpha x) =\alpha Sx$ and $S(x+y)=Sx +Sy$ holds for all $\alpha\in\R$ and $x,y\in E$ is called a linear operator; when the image of the closed unit ball on $E$ via $S$ is bounded, we call $S$ bounded.
The dual spaces $E^*$ and $F^*$ are the spaces of bounded linear functionals on $E$ and $F$, respectively. The evaluation of a dual element $x^*\in E^*$ at a point $x \in E$ is written as $\fip{x^*,x}{E^*,E}\coloneq x^*(x)$.
 Given a bounded linear operator $S:E\to F$, its adjoint operator $S^*:F^*\to E^*$ is defined as $\fip{S^*y^*,x}{E^*,E}=\fip{y^*,Sx}{F^*,F}$ for all $x \in E$ and $y^* \in F^*$.

\tb{Probability:}
The space of finite signed (resp.\ probability) measures on $(\X,\mB_\X)$ is denoted by $\M_b(\X)$ [resp.\ $\M_1^+(\X)$]. For any $\mF\subseteq \M_b(\X)$, we denote by $[\mF]^0$ the set of $\F\in\mF$ for which $\F(\X)=0$.
For a point $x\in \X$, the Dirac measure $\delta_x\in\M_1^+(\X)$ is defined as
\begin{align}
    \delta_x(A) =\begin{cases} 1 &\text{ if } x\in A,\\
    0&\text{ if } x\not \in A,\end{cases} \quad \text{ for }  A\in\mB_\X.
\end{align}
For $\P\in\M_1^+(\X)$, its support is defined as the smallest closed set whose complement has null $\P$-measure: a closed set $C\in \mathcal B_{\X}$ is the support of $\P$ if (i) $\P(\X\setminus C)=0$ and (ii) for any closed set $C_1\subseteq \X$ for which $\P(\X\setminus C_1)=0$, $C\subseteq C_1$ holds.\footnote{Note that the support of a Borel measure exists \citep[Proposition 2.3]{kozarzewski18existence} if the topology of $\X$ has a countable basis (in other words, $\X$ is second-countable); in metric spaces, the notion of second countability and separability coincide \citep[Proposition 2.1.4]{dudley04real}. \label{foot:2ndcount}}
Let $\P,\Q\in\M_1^+(\X)$, and let $\P$ be absolutely continuous w.r.t.\ $\Q$ ($\P\ll \Q$); we denote the corresponding Radon-Nikodym derivative by $\tfrac{\d \P}{\d \Q}$. Note that the Radon-Nikodym derivative is unique $\Q$-a.e.\ \citep[Theorem 5.5.4]{dudley04real} and can be extended to the case of $\P\in \M_b(\X)$ and $\Q$ being a finite non-negative measure \citep[Corollary 5.6.2]{dudley04real}.
Denote by $\F^+$ and $\F^-$ the non-negative measures of the Jordan decomposition of $\F\in\M_b(\X)$ and let $|\F| = \F^+ + \F^-$ denote the total variation of $\F$. Let $\H$ be a Hilbert space. The Bochner integral of a measurable function $f:\X\to\H$ w.r.t.\ $\F\in\M_b(\X)$ is defined as
\begin{equation}
\int_{\X} f(x)\d\F(x) = \int_{\X} f(x)\d\F^+(x) - \int_{\X} f(x) \d \F^-(x);
\end{equation}
it exists iff $\int_\X \norm{f(x)}{\X} \d|\F|(x)< \infty$ \citep[Chapter~II.2, Theorem~2]{diestel77vector}. Let $X \sim \P \in \mathcal M_1^+(\H)$. The expectation of $X$ is $\E_\P[X] \ =\ \int_{\H} x\d \P(x)$; the integral is meant in Bochner's sense.

\tb{Reproducing kernel Hilbert space:}
A function $K:\X^2\to\R$ is called a kernel if there exists a Hilbert space $\H$ and a feature map $\Phi:\X\to\H$ such that $K(x,x') = \fip{\Phi(x),\Phi(x')}{\H}$ for all $x,x'\in\X$. A Hilbert space $\H_K$ of $\X \to \R$ functions is called reproducing kernel Hilbert space (RKHS) associated to a kernel $K:\X^2\to \R$ if $K(\cdot,x)\in \H_{K}$ and $\fip{K(\cdot,x),f}{\H_{K}}=f(x)$ for all $x\in\X$ and $f\in\H_K$, where the latter relation is referred to as the reproducing property; $K(\cdot,x)$ denotes the function $x'\mapsto K(x',x)$ ($x\in\X$) and it is called the canonical feature map (of $K$); $K$ is called the reproducing kernel of $\H_K$.
A kernel is equivalent to a reproducing kernel, and the correspondence of kernels and RKHSs is one-to-one. The norm in $\H_K$ is denoted by $\norm{f}{\H_K} = \sqrt{\fip{f,f}{\H_K}}$ ($f \in \H_K$). We assume that all kernels in this manuscript are Borel-measurable.
Let $K:\X^2\to\R$ be a kernel.
For $p\in[1,\infty)$, let the set of finite signed  measures with finite $p$-th moments w.r.t.\ $K(\cdot,x)$ be
\begin{align}
\MoveEqLeft\mPb^p(K;\X) = \\[-2.5ex]
&=\Bigg\{ \F\in\M_{b}(\X):\int_{\X} \overbrace{\left\|K(\cdot,x)\right\|_{\H_K}^p}^{=K^{\frac{p}{2}}(x,x)}\d |\F|(x)<\infty\Bigg\},\notag
\end{align}
and the subset of probability measures $\mP^p(K;\X) = \mPb^p(K;\X)\cap \M_1^+(\X)$; we omit the exponent for both sets when $p=1$. Note that (i) $\F \in \mPb^p(K;\X) \implies |\F|\in \mPb^p(K;\X)$ [as $\F \in \M_b(\X) \implies |\F| \in \M_b(\X)$] and (ii) $\mPb^q(K;\X)\subseteq \mPb^{p}(K;\X)$ for any $q\geq p\ge 1$, as the $|\F|$-integrability of $\norm{K(\cdot,x)}{\H_K}^q$ implies the $|\F|$-integrability of $\norm{K(\cdot,x)}{\H_K}^p$ when $\F$ is finite.
The mean embedding w.r.t.\ $K$ is defined as
\begin{align}
    \mu_K: \mPb(K;\X) &\to \H_K, && \F \mapsto  \int_\X K(\cdot, x)\d \F(x).
\end{align}
We say that $K$ is integrally strictly positive definite (i.s.p.d.) to $\mF\subseteq \mPb(K;\X)$ (or simply, $\mF$-i.s.p.d.)
(\citealt[Definition 1]{szabo18characteristic2}; \citealt[Definition 5]{simon-gabriel18kernel}; \citealt[Definition 2.1]{steinwart21strictly})
if
\begin{equation}
    \norm{\mu_K(\F)}{\H_K}^2 = \int_\X\int_\X K(x,x')\d\F(x)\d\F(x')>0, \label{eq:def:ispd}
\end{equation}
for any $\F\in\mF\setminus\{0\}$.
$K$ is called characteristic to $\mF\subseteq \mPb(K;\X)$ (or simply $\mF$-characteristic) (\citealt[page 2270]{sejdinovic13equivalence}; \citealt[Definition 5]{simon-gabriel18kernel}) if $\mu_K$ is injective on $\mF$, in other words when $\mu_K(\F) = \mu_K(\F')$ iff $\F=\F'$ for all $\F,\F'\in\mF$.\footnote{If $K$ is bounded (in other words, $\sup_{x,x'\in \X}K(x,x') <\infty$), one has that $\mP(K;\X)=\M_1^+(\X)$, and, when taking $\mF=\M_1^+(\X)$, one gets back the (classical) characteristic property of kernels \citep{fukumizu08kernel2,sriperumbudur10hilbert}.}  $K$ is referred to as characteristic to $\mF\subseteq\mPb(K;\X)$ w.r.t.\ $\F$ [or simply $(\mF,\F)$-characteristic; \citealt[page~7]{barp24targeted}] when $\mu_K(\F)=\mu_K(\F')$ iff $\F=\F'$ for any $\F'\in\mF$.\footnote{When $K$ is a Stein kernel and $\mF=\mP(K;\X)$, one recovers the notion introduced by \citet[page 7]{cribeiro26minimax2}.\label{footnote:F-char-wrt-F}}
For a $p\in [1,\infty)$, $\P\in\mP^p(K;\X)$, and $\H_K$ separable\footnote{A continuous kernel $K$ on a separable topological domain $(\X,\tau_\X)$ is known \citep[Lemma~4.33]{steinwart08support} to be sufficient for $\H_K$ to be separable.\label{footnote:separable-Hk}}, consider the inclusion operator $\id:\H_K\to L_p(\X,\P)$, $h\mapsto [h]_\sim$. It is known \citep[Theorem~4.26]{steinwart08support} that $\id$ is continuous and $ \norm{\id}{} \leq \left[\int_{\X}K^{p/2}(x,x)\d\P(x)\right]^{1/p}$,
where the norm on the l.h.s.\ is the operator norm. Furthermore, the adjoint of this inclusion is the operator $S_K:L_{p'}(\X,\P)\to\H_K$ defined as \begin{align}\label{eq:S_K}
    (S_K g)(x)=\int_\X K(x,y)g(y)\d\P(y),
\end{align}
where $g\in L_{p'}(\X,\P)$ and $x\in\X$. We say that the kernel $K$ is $L_p$-universal if $\H_K$ is dense in $L_p(\X,\P)$ for all $\P\in\mP^p(K;\X)$.\footnote{When $K$ is bounded, $\mP^p(K;\X) = \M_1^+(\X)$, and one gets back the traditional notion of $L_p$-universality as introduced by \citet{carmeli10vector,sriperumbudur11universality}.}
For $p\in[1,\infty)$ and $\P\in\mP^p(K;\X)$, we denote the set of finite signed measures absolutely continuous w.r.t.\ $\P$ and with $L_{p'}$-integrable Radon-Nikodym derivative as
\begin{align}
\Abs_{p'}(\P) & = \left\{\F\in\M_b(\X): \F\ll \P, \frac{\d \F}{\d \P} \in L_{p'}(\X,\P)\right\}; \notag
\end{align}
the union of all $\Abs_{p'}(\P)$ with $\P\in\mP^p(K;\X)$ is denoted as
\begin{align}
\A^p(K;\X) & = \cup_{\P\in\mP^p(K;\X)} \Abs_{p'}(\P);
\end{align}
we omit the exponent when $p=1$.

\tb{(Semi)-metric space:} Let $\X$ be a non-empty set and let
$\rho:\X^2\to[0,\infty)$ be a map such that (i) $\rho(x,y)=0$ iff $x=y$ and (ii) $\rho(x,y) = \rho(y,x)$ for all $x,y\in\X$. Then, $(\X,\rho)$ is said to be a semimetric space and $\rho$ is called a semimetric on $\X$. If $\rho$ further fulfills (iii) $\rho(x,y)\leq \rho(x,z) + \rho(z,y)$ for all $x,y,z\in\X$, we call $(\X,\rho)$ a metric space and $\rho$ a metric on $\X$. Unless otherwise specified, throughout the manuscript, $(\X,\rho)$ stands for a semimetric space. Let $p\in[1,\infty)$, and define $\mPb^p(\rho;\X)=\{\F\in\M_b(\X):\text{ there exists } x_0\in\X \text{ s.t.\ } \int_\X \rho^{p/2}(x,x_0)\ \d|\F|(x)<\infty\}$\footnote{Note that $\mPb^q(\rho;\X)\subseteq \mPb^{p}(\rho;\X)$ for any $q\geq p$, as the $|\F|$-integrability of $\rho^q(x,x_0)$ for some $x_0\in\X$ implies the $|\F|$-integrability of $\rho^p(x,x_0)$ when $\F$ is finite.\label{footnote:P-b-q-rho:inclusion}}
 and $\mP^p(\rho;\X) = \mPb^p(\rho;\X)\cap \M_1^+(\X)$; we omit the exponent when $p=1$.
We say that $(\X,\rho)$ has negative type (n.t.) iff for all $n\in\Np$ such that $n\geq 2$, $\{x_i\}_{i=1}^n\subseteq\X$, and $\{\alpha_i\}_{i=1}^n\subset\R$ s.t. $\sum_{i=1}^n \alpha_i = 0$, we have $\sum_{i,j\in [n]} \alpha_i\alpha_j\rho(x_i,x_j) \leq 0$.
Further, we say that $(\X,\rho)$ has strong negative type (s.n.t.) to $\mF \subseteq \mPb^{2}(\rho;\X)$ (or simply $\mF$-s.n.t.) if $(\X,\rho)$ has n.t.\ and
\begin{align}
    \int_{\X}\int_{\X}\rho (x,y) ~\d\F(x)\d\F(y)<0, \label{eq:snt}
\end{align}
for all $\F \in \mF\setminus\{0\}$.\footnote{When $\mF=[\mPb^2(\rho;\X)]^0$ we recover the classical notion of s.n.t.\ space \citep[Definition 28]{sejdinovic13equivalence}. \label{footnote:snt-specialization} } We say that $\rho$ has been induced by a kernel $K:\X^2\to\R$ if for all $x,y\in \X$, one has
\begin{align}
\rho(x,y) &= K(x,x)+ K(y,y) - 2K(x,y) \label{eq:K2rho}
\end{align}
\citep[Corollary 16]{sejdinovic13equivalence}; in this case we write $\rho_K\coloneq \rho$.\footnote{$\rho_K$ is a valid semimetric iff  $K$ is non-degenerate, in other words if the map $x\mapsto K(\cdot,x)$ is injective \citep[Proposition 14]{sejdinovic13equivalence}; in this case $\rho_K$ is even of negative type \citep[Corollary~16]{sejdinovic13equivalence}.\label{footnote:rho-K-semi-metric:sufficient-condition}} Further, we say that a kernel $K:\X^2 \to \R$ is induced by the semimetric $\rho$ of negative type \citep[Definition 13]{sejdinovic13equivalence} if there exists some $x_0\in\X$ such that for all $x,x'\in \X$, one has
\begin{align}
    K(x,x') = \frac{1}{2}[\rho(x,x_0) + \rho(x',x_0) - \rho(x,x')]; \label{eq:K_rho}
\end{align}
in this case we write $K_\rho\coloneq K$.
Note that for $p\in\Np$ one has that $\mPb^p(K;\X)=\mPb^p(\rho;\X)$ when $K$ has been induced by $\rho$, and vice versa \citep[Propositions 14 and 20]{sejdinovic13equivalence}.
If a semimetric space $(\X,\rho)$ has n.t., then $(\X,\rho^{1/2})$ is a metric space of n.t.\ \citep[page~2266]{sejdinovic13equivalence}. Let $\H$ be a Hilbert space of functions; we call an injective map $\phi:(\X,\rho^{1/2})\to\H$ to be an isometric embedding of $(\X,\rho^{1/2})$ into $\H$ if $\rho(x,y)=\norm{\phi(x)-\phi(y)}{\H}^2$ for all $x,y\in\X$;
when $(\X,\rho)$ has n.t.\ such a $\phi$ always exists \citep[Proposition 3]{sejdinovic13equivalence}.
For any isometric embedding $\phi$, consider the map
\begin{align}
\beta_\phi:\F\in\mPb(\rho;\X)\mapsto \int_\X\phi(x)\d\F(x)\in\H, \label{eq:beta_phi}
\end{align}
where $\beta_\phi(\F)$ is known as the barycenter of $\F$ w.r.t.\ $\phi$.\footnote{Note that the r.h.s.\ of \eqref{eq:beta_phi} always exists for $\F\in\mPb(\rho,\X)$. Indeed the Bochner integral exists iff $\int_\X\norm{\phi(x)}{\H} \d |\F|(x) <\infty$. Taking $x_0\in\X$ s.t.\ $\int_\X\rho^{1/2}(x,x_0)\d |\F|(x)<\infty$ [guaranteed to exist as $\F\in\mPb(\rho;\X)]$ we have that $\int_\X\norm{\phi(x)}{\H}\d|\F|(x) = \int_\X\norm{\phi(x) + \phi(x_0) - \phi(x_0)}{\H}\d|\F|(x) \jover{a}{\leq} \int_\X\norm{\phi(x)-\phi(x_0)}{\H}\d|\F|(x) + \int_\X\norm{\phi(x_0)}{\H}\d|\F|(x) \jover{b}{=}  \int_\X\rho^{1/2}(x,x_0)\d|\F|(x) + \norm{\phi(x_0)}{\H}|\F|(\X) \jover{c}{<}\infty$, where (a) comes from the triangular inequality combined with the monotonicity and linearity of integration, (b) holds as $\phi$ is an isometric embedding, while (c) is by the choice of $x_0$ and as $\F$ (and hence $|\F|$) is a finite measure.} Note that when $K$ is non-degenerate, the canonical feature map $x \mapsto K(\cdot,x)$ is an isometric embedding of $(\X,\rho_K^{1/2})$ into the RKHS $\H_K$ \citep[page~2277]{sejdinovic13equivalence}.

\section{RESULTS} \label{sec:results}

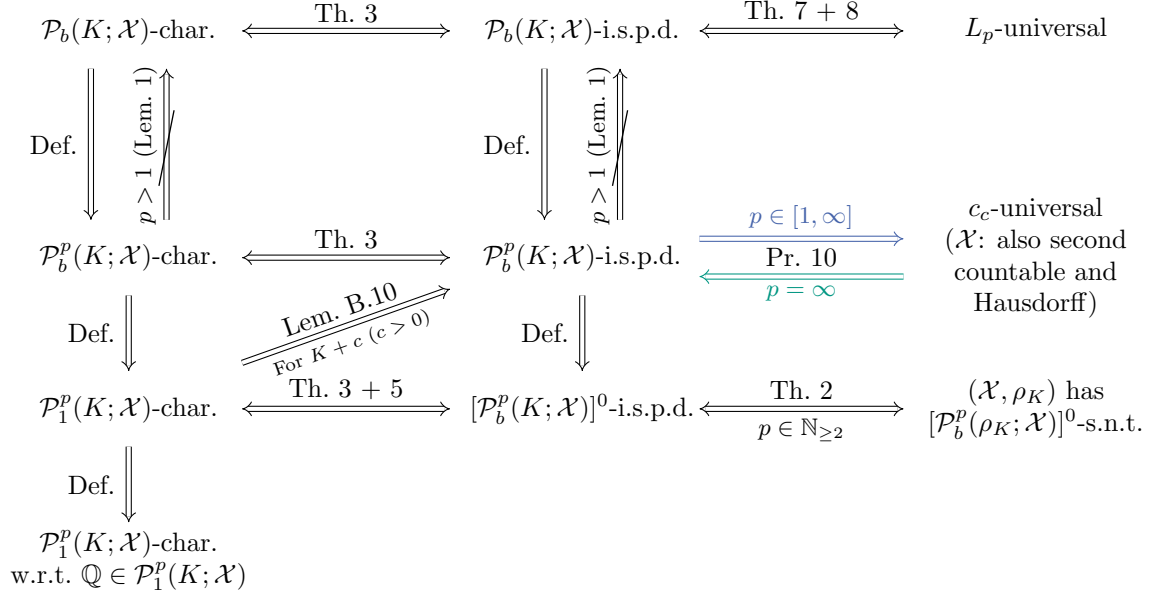
\begin{figure*}
\centering
\begin{tikzpicture}[
    imp/.style    = {-{implies},       double equal sign distance},  %
    revimp/.style = {{implies}-,       double equal sign distance},  %
    iff/.style    = {implies-implies,  double equal sign distance},  %
    nd/.style     = {align=center},                        %
    deflab/.style = {midway, left, xshift=-1.5pt},         %
    vlab/.style   = {midway, above, sloped, rotate=180},   %
  ]
  \pgfmathsetmacro{\xL}{0.25+\shift+\aperture}    %
  \pgfmathsetmacro{\xM}{5.25+\shift}              %
  \pgfmathsetmacro{\xR}{10.25+\shift-\aperture}   %
  \pgfmathsetmacro{\xLa}{1.75+\shift+\aperture}   \pgfmathsetmacro{\xLb}{3.5+\shift}
  \pgfmathsetmacro{\xMa}{6.8+\shift}              \pgfmathsetmacro{\xMb}{8.5+\shift-\aperture}
  \def\yQ{0.5}  \def\yA{2.5}  \def\yB{4.5}  \def\yC{7.5}

  \node[nd] (pCharQ)  at (\xL,\yQ) {$\mP^{p}(K;\X)$-char.\\ w.r.t.\ $\Q\in\mP^p(K;\X)$};
  \node[nd] (pChar)   at (\xL,\yA) {$\mP^{p}(K;\X)$-char.};
  \node[nd] (pbpChar) at (\xL,\yB) {$\mPb^{p}(K;\X)$-char.};
  \node[nd] (pbChar)  at (\xL,\yC) {$\mPb(K;\X)$-char.};
  \node[nd] (pb0Ispd) at (\xM,\yA) {$[\mPb^{p} (K;\X)]^0$-i.s.p.d.};
  \node[nd] (pbpIspd) at (\xM,\yB) {$\mPb^{p}(K;\X)$-i.s.p.d.};
  \node[nd] (pbIspd)  at (\xM,\yC) {$\mPb(K;\X)$-i.s.p.d.};
  \node[nd] (snt)     at (\xR,\yA) {$(\X,\rho_{K})$ has\\ $[\mPb^{p} (\rho_K;\X)]^0$-s.n.t.};
  \node[nd] (ccUniv)  at (\xR,\yB) {$c_c$-universal\\ ($\X$: also second \\ countable and\\ Hausdorff)};
  \node[nd] (LpUniv)  at (\xR,\yC) {$L_p$-universal};

  \draw[imp] (\xL,\yA-0.5) -- node[deflab] {Def.} (\xL,\yQ+0.5);
  \draw[imp] (\xL,\yB-0.5) -- node[deflab] {Def.} (\xL,\yA+0.5);
  \draw[imp] (\xM,\yB-0.5) -- node[deflab] {Def.} (\xM,\yA+0.5);

  \foreach \x in {\xL,\xM}{
    \draw[imp] (\x-0.5,\yC-0.5) -- node[deflab] {Def.} (\x-0.5,\yB+0.5);
    \draw[revimp] (\x+0.5,\yC-0.5) -- node[vlab] {\small{$p>1$} (Lem.~\ref*{lem:monotonicity_mpb})} (\x+.5,\yB+0.5);
    \draw[line width=.5pt] (\x+.4,\yB+0.95) -- (\x+0.6,\yB+1.95);  %
  }

  \draw[iff] (\xLa,\yA) -- node[above] {Th.~\ref{th:ispd0_iff_char} + \ref{th:char_p1_iff_pb0}} (\xLb,\yA);
  \draw[iff] (\xLa,\yB) -- node[above] {Th.~\ref{th:ispd0_iff_char}} (\xLb,\yB);
  \draw[iff] (\xLa,\yC) -- node[above] {Th.~\ref{th:ispd0_iff_char}} (\xLb,\yC);

  \draw[imp] ([shift={(1.5,.3)}]pChar.north) -- node[midway,above,sloped] {Lem.~\ref*{lem:kc_ispd}} node[midway,below,sloped] {\scriptsize{For $K+c$ ($c>0$)}} ([shift={(-1.75,-0.1)}]pbpIspd.south);

  \draw[iff] (\xMa,\yA) -- node[above] {Th.~\ref{th:snt_iff_ispd}} node[below] {\small{$p\in\N_{\ge 2}$}} (\xMb,\yA);
  \draw[imp,KITblue] (\xMa,\yB+0.25) -- node[above] {\small{$p\in[1,\infty]$}} node[below] {\textcolor{black}{Pr.~\ref{prop:cc-univ}}} (\xMb,\yB+0.25);
  \draw[imp,KITgreen] (\xMb,\yB-0.25) -- node[below] {\small{$p=\infty$}} (\xMa,\yB-0.25);
  \draw[iff] (\xMa,\yC) -- node[above] {Th.~\ref{th:univ} + \ref{th:Ap=Pb}} (\xMb,\yC);
\end{tikzpicture}
\caption{Implications of our results for a continuous kernel $K$ on a separable topological space $\X$. Def.\ indicates that it follows from the definition. Char., Lem., Pr., and Th.\ stand for characteristic, lemma, proposition, and theorem, respectively.}
\label{fig:summary}
\end{figure*}

This section is dedicated to our results on the interaction of $\mF$-characteristic, $\mF$-i.s.p.d., $\mF$-s.n.t.\ properties, and $L_p$-universality for general kernels (allowing unbounded ones) and various measure families $\mF$; see Figure~\ref{fig:set_relations} for a visual summary of the relations with the most relevant $\mF$ families in this article.
Before delving into the details of these results, it is instrumental to discuss how the sets $\mPb^p(K;\X)$ and $\mP^p(K;\X)$ [with $p\in[1,\infty)$]---which the previous expressivity concepts heavily rely on---(i) arise naturally when dealing with mean embedding-based ITMs, and (ii) behave fundamentally differently for bounded and unbounded kernels, making their analysis intricate.

Indeed, the sets $\mP^p(K;\X)$ are the natural domains of kernel statistics: having finite $p$-th moment is tied to the existence of different kernel-based ITMs like MMD, HSIC, and KSD, as we elaborate in the following. We define the maximum mean discrepancy $\MMD$ as the mapping
\begin{align}
    (\P,\Q)\mapsto \norm{\mu_K(\P)-\mu_K(\Q)}{\H_K}.
\end{align}
In order for $\MMD$ to be well-defined on $(\P,\Q)\in \left[\M_1^+(\X)\right]^2$, the mean embedding of both probability measures must exist, in other words, one needs that $\P,\Q\in\mP(K;\X)$.

As another example, for $M\ge2$, HSIC is defined as the mapping
\begin{align}
    \P \mapsto \norm{\mu_{K}(\P) - \mu_{K}(\otimes_{m=1}^M\P|_m)}{\H_{K}},
\end{align}
where $\P\in\M_1^+(\times_{m=1}^M\X_m)$, $\{\P|_m\}_{m=1}^M$ are the marginals of $\P$,
$\otimes_{m=1}^M\P|_m$ is the product measure of these marginals, and $K=\otimes_{m=1}^MK_m$ is the product kernel [$(\otimes_{m=1}^MK_m)(x,y)\coloneq\prod_{m=1}^M K_m(x_m,y_m)$; $x=(x_m)_{m=1}^M,y=(y_m)_{m=1}^M\in\times_{m=1}^M \X_m$]. Therefore,  HSIC is well-defined   iff $\P,\, \otimes_{m=1}^M \P|_m\in \mP(K; \times_{m=1}^M\X_m)$; when $M=2$, $\P|_1\in\mP^2(K_1;\X_1)$ and $\P|_2\in \mP^2(K_2;\X_2)$ is a sufficient condition for HSIC to be well-defined \citep[page 2271]{sejdinovic13equivalence}. A similar higher-order sufficient condition can be derived for $M\ge 2$. Indeed, in this case
\begin{align}
    \MoveEqLeft  \int_{\times_{m=1}^M\X_m} \norm{K(\cdot,x)}{\H_{K}} \d \P(x)\\
    &\jover{a}{=} \int_{\times_{m=1}^M\X_m} \sqrt{K(x,x)} \d \P(x)\\
    & \jover{b}{=} \int_{\times_{m=1}^M\X_m} \prod_{m=1}^MK^{\frac{1}{2}}_m(x_{m},x_{m}) \d \P(x)
   \\& \jover{c}{\le} \prod_{m=1}^M \left[\int_{\X_m}K_m^{\frac{M}{2}}(x_m, x_m) \d \P|_m(x_m)\right]^{1/M}\hspace*{-0.1cm}, \label{eq:hsic_mge2}
\end{align}
where (a) follows from the fact that in a Hilbert space the norm is induced by its inner product and the reproducing property, (b) comes from the definition of product kernel and (c) holds by the generalized Hölder's  inequality (recalled in Lemma~\ref{lem:hoelder}, taking $r \leftarrow 1$ and $p_m \leftarrow M$ for all $m\in[M]$). Therefore, $\P|_m\in \mP^M(K_m;\X_m)\subseteq \mP(K_m;\X_m)$ for all $m\in[M]$ implies that $\otimes_{m=1}^M \P|_m\in \mP(K; \times_{m=1}^M\X_m)$, and hence \eqref{eq:hsic_mge2} shows that  $\P\in\mP(K;\X)$.

As the last example, we consider a kernel $K_{\P_0}$ such that $\E_{\P_0}\big[K_{\P_0}(\cdot,X)\big]=0$. $\KSD$ is defined as the mapping
\begin{align}
    \P \mapsto \norm{\mu_{K_{\P_0}}(\P)}{\H_{K_{\P_0}}},
\end{align}
where $\P\in\M_1^+(\X)$, and as before, it is well-defined as long as $\mu_{K_{\P_0}}(\P)$ exists, in other words, if $\P\in \mP(K_{\P_0};\X)$. Imposing instead the stricter condition of $\P\in\mP^2(K_{\P_0};\X)$ allows, for example, to obtain sharper minimax lower bounds for U-statistic-based KSD estimators \citep{gogolashvili26minimax}.

Having illustrated the need for understanding the sets $\mP^p(K;\X)\subsetneq \mPb^p(K;\X)$ [$p\in[1,\infty)$], our first result [summarized in Figure~\ref{fig:set_relations}(b)] sheds light on a sharp contrast in the behavior of bounded and unbounded kernels: for bounded kernels the mapping $p\mapsto \mPb^p(K;\X)$ is constant, while for unbounded kernels it is strictly non-increasing w.r.t.\ set inclusion, in other words $\mPb^q(K;\X)\subsetneq \mPb^p(K;\X)$ for $q>p$.

\begin{lemma}[Relationship between the $\mP^p(K;\X)$ and $\mPb^p(K;\X)$ sets; see Appendix~\ref{subsec:proof_monotonicity}]\label{lem:monotonicity_mpb}
    Let $(\X,\tau_\X)$ be a topological space
    with a kernel $K:\X^2\to\R$, and $1\leq p<q<\infty$.
    \begin{enumerate}[label=\roman*)]
        \item If $K$ is bounded, then  $\mP^q(K;\X)=\mP^p(K;\X)=\M_1^+(\X)$ and $\mPb^q(K;\X)=\mPb^p(K;\X)=\M_b(\X)$,
        \item if $K$ is unbounded, then $\mP^q(K;\X)\subsetneq \mP^p(K;\X)$ and $\mPb^q(K;\X)\subsetneq \mPb^p(K;\X)$.
    \end{enumerate}
\end{lemma}

We proceed by presenting our main results on the relations of
the $\mF$-characteristic, $\mF$-i.s.p.d.\ and $\mF$-s.n.t.\ properties (Section~\ref{sec:char-ispd-snt}), and $L_p$-universality (Section~\ref{sec:Lp-universality}), with a summary of the results in Table~\ref{tab:results-in-the-main} and an illustration of the key implications in Figure~\ref{fig:summary}.

\begin{figure*}[t]
\centering
\begin{subfigure}{.449\linewidth}\centering
    \begin{tikzpicture}[scale = 0.25]
   \def\firstellipse{(0,0) ellipse (13 and 8)}
   \def\secondellipse{(4,0) ellipse (8 and 6)}
   \def\thirdellipse{(-3.5,0) ellipse (8 and 6)}

  \begin{scope}
  \fill[gray5, fill opacity=0.1] \firstellipse;
  \fill[gray5, fill opacity=0.4] \secondellipse;
  \fill[gray5, fill opacity=0.4] \thirdellipse;
  \draw \firstellipse;
  \draw \secondellipse;
  \draw \thirdellipse;
  \end{scope}

   \node at (.25,1) (A) {$[\mPb^p(K;\X)]^0$};
   \node at (-6.5,2.7) (B) {$\mPb^p(K;\X)$};
   \node at (8.75,0) (C) {$[\M_b(\X)]^0$};
   \node at (-10,7) (D) {$\M_b(\X)$};

    \end{tikzpicture}
\caption{$\mPb^p(K;\X)$ and $[\mPb^p(K;\X)]^0$.}
\end{subfigure}~
    \begin{subfigure}{.440\linewidth}\centering
    \begin{tikzpicture}[scale = 0.25]
   \def\firstellipse{(0,0) ellipse (11 and 8)}
   \def\secondellipse{(1,0) ellipse (10 and 7)}
   \def\thirdellipse{(4,0) ellipse (7 and 5)}
   \def\fourthellipse{(9.5,0) ellipse (1.5 and 1)}
   \def\fifthellipse{(5.5,0) ellipse (5.5 and 3.5)}

  \begin{scope}
  \fill[gray5, fill opacity=0.1] \firstellipse;
  \fill[gray5, fill opacity=0.4] \secondellipse;
  \fill[gray5, fill opacity=0.4] \thirdellipse;
  \fill[gray5, fill opacity=0.4] \fifthellipse;
  \fill[black, fill opacity=0.5] \fourthellipse;
  \draw \firstellipse;
  \draw \secondellipse;
  \draw \thirdellipse;
  \draw \fourthellipse;
  \draw \fifthellipse;
  \end{scope}

   \node at (13.25,0) (A) {$\mPb^\infty(\X)$};
   \node at (-1.25,0) (B) {$\cdots$};
   \node at (-6,0) (C) {$\mPb^2(K;\X)$};
   \node at (-10,7) (D) {$\mPb(K;\X)$};
   \node at (3,0) (F) {$\mPb^p(K;\X)$};
   \node at (7,0) (B) {$\cdots$};

    \end{tikzpicture}
\caption{$\mPb^p(K;\X)$ and $\mPb^q(K;\X)$, with $q\geq p$.}
\end{subfigure}
\\
\caption{Relationships of the different sets relevant to our results. The set $\mPb^\infty(K;\X)$ is defined in \eqref{eq:Pb-infty} and discussed afterwards.}
\label{fig:set_relations}
\end{figure*}
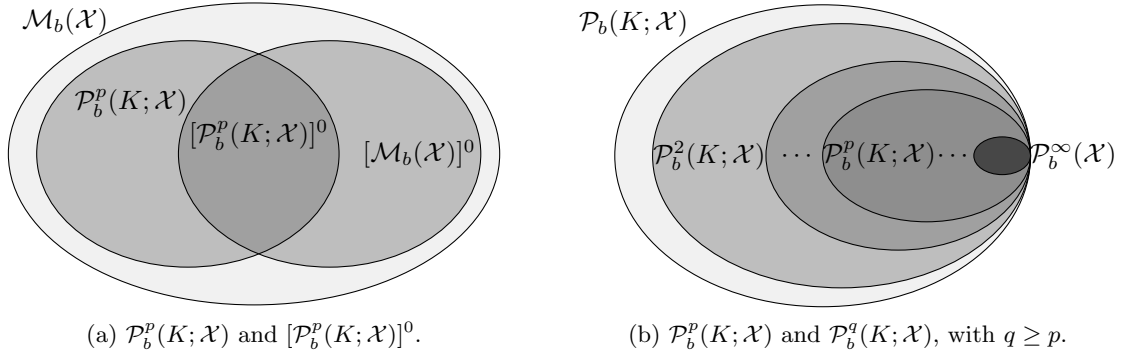

\subsection{Characteristic, I.s.p.d., and S.n.t.\ Properties} \label{sec:char-ispd-snt}
Our first theorem equates a kernel being i.s.p.d.\ with the induced metric space having s.n.t.\ to the same set, which we will use later to provide a slight extension of a result by \citet{sejdinovic13equivalence} from $\mF = [\mPb^{2}(K;\X)]^0$ to any $\mF\subseteq [\mPb^{2}(K;\X)]^0$ [see Remark~\ref{rem:6}\ref{rem:6:2}].

\begin{theorem}[$\mF$-s.n.t.\ vs $\mF$-i.s.p.d.; see Appendix~\ref{subsec:proof_snt}]\label{th:snt_iff_ispd}
Let $(\X,\rho_K)$ be a semimetric space\textsuperscript{\ref{footnote:rho-K-semi-metric:sufficient-condition}} associated to a kernel $K:\X^2 \to \R$ and $\mF\subseteq [\mPb^{2}(K;\X)]^0$. Then, $(\X,\rho_K)$ has $\mF$-s.n.t.\ iff $K$ is $\mF$-i.s.p.d.
\end{theorem}

Our next result clarifies the relationship between a kernel being $\mF$-characteristic and $\mF$-i.s.p.d.\ for $\mF\subseteq \mPb^p(K;\X)$.

\begin{theorem}[$\mF$-i.s.p.d.\ vs $\mF$-characteristic; see Appendix~\ref{subsec:proof_ispd0_iff_char}]\label{th:ispd0_iff_char}
    Let $(\X,\tau_\X)$ be a topological space, $K:\X^2\to\R$ a kernel, and $\mF\subseteq \mPb(K;\X)$ such that (i) $0\in\mF$. Then, $K$ being $\mF$-characteristic implies that  $K$ is $\mF$-i.s.p.d.
    If $\mF$ is closed w.r.t.\ (ii) negation  ($\F\in\mF\implies -\F\in\mF$) and (iii) addition  ($\F_1,\F_2\in\mF \implies \F_1+\F_2\in\mF$), the converse holds.
\end{theorem}
We note the following.

\begin{remark}\label{rem:4}~

\begin{enumerate}[label=(\alph*)]
 \item \tb{Group structure of $\mF$}:  $\mF$-s satisfying conditions (i)-(iii) in Theorem~\ref{th:ispd0_iff_char} form subgroups of $(\M_{b}(\X),+)$; examples include $\mPb^p(K;\X)$, $[\mPb^p(K;\X)]^0$, and $\Abs_{p'}(\P)$ (by Corollary~\ref{lem:abs_holds_ii-iii}). \label{remark4:(a)}

 \item \tb{Relationship to existing results}: Theorem~\ref{th:ispd0_iff_char} generalizes a known result by \citet[Theorem 6(ii)--(iii)]{simon-gabriel18kernel} from $\mF$ being a Banach space to $\mF$ being a group with addition. \label{remark4:(b)}

 \item \tb{Sufficient condition for s.n.t.}: For $p\in[2,\infty)$, a space $(\X,\rho)$ with n.t.\ has $[\mPb^p(\rho;\X)]^0$-s.n.t.\  iff $\beta_\phi$ is injective on $\mP^p(\rho;\X)$ (see Lemma~\ref{lem:inj_barycenter}).\footnote{Note that it is sufficient to have one $\phi$ such that $\beta_\phi$ is injective to prove that $(\X,\rho)$ has s.n.t.\ Indeed, this can be shown by using the fact that any two isometric embeddings $\phi_1:(\X,\rho^{1/2})\to\H_1$ and $\phi_2:(\X,\rho^{1/2})\to\H_2$ are equivalent up to isometries \citep[page 3292]{lyons13distance}.} For $p=2$, one gets back the statement by \citet[Proposition 3.1]{lyons13distance}.
 \end{enumerate}
\end{remark}

Our next result shows that for a kernel being characteristic to $\mP^p(K;\X)$ and to $[\mPb^p(K;\X)]^0$ are equivalent.

 \begin{theorem}[$\mP^p(K;\X)$-characteristic vs {$[\mPb^p(K;\X)]^0$}-characteristic; see Appendix~\ref{subsec:proof_char_p1_iff_pb0}]\label{th:char_p1_iff_pb0}
     Let $(\X,\tau_\X)$ be a topological space, $K:\X^2\to\R$ a kernel, and $p\in [1,\infty)$. Then, $K$ is $\mP^p(K;\X)$-characteristic iff $K$ is $[\mPb^p(K;\X)]^0$-characteristic.
 \end{theorem}

We elaborate on this result in the following remark.
\begin{remark}~
\begin{enumerate}[label=(\alph*)]\label{rem:6}
\item \tb{Relationship to existing results:} Together with Theorem~\ref{th:ispd0_iff_char} [with the choice $\mF=\left[\mPb^p(K;\X)\right]^0$, see also Remark~\ref{rem:4}\ref{remark4:(a)}], Theorem~\ref{th:char_p1_iff_pb0} shows that a  kernel $K$ is $\mP^p(K;\X)$-characteristic iff it is $\left[\mPb^p(K;\X)\right]^0$-i.s.p.d.
This generalizes a known result by \citet[Lemma 8]{sriperumbudur10hilbert} and \citet[Theorem 8(i)--(ii)]{simon-gabriel18kernel} from bounded to unbounded kernels with finite $p$-th moment.\label{rem:6:1}

\item \tb{Characteristic kernels induce s.n.t.\ spaces and vice-versa}: Let $p\in\N_{\ge 2}$.
Note that Theorem~\ref{th:char_p1_iff_pb0} with Theorem~\ref{th:snt_iff_ispd}--\ref{th:ispd0_iff_char} (choosing $\mF=\mPb^p(K;\X)]^0$) and Proposition~\ref{prop:equiv_int_k_rho} imply that $K$ is $\mP^p(K;\X)$-characteristic iff its induced semimetric space has $[\mPb^p(\rho_K;\X)]^0=[\mPb^p(K;\X)]^0$-s.n.t. In the specific case of $p=2$, this result reduces to that by  \citet[Proposition 29]{sejdinovic13equivalence}. Conversely, Theorems~\ref{th:snt_iff_ispd}--\ref{th:ispd0_iff_char} (with substitutions $\mF\leftarrow[\mPb^p(\rho;\X)]^0= [\mPb^p(K_\rho;\X)]^0  $, $K \leftarrow K_\rho$) with Theorem~\ref{th:char_p1_iff_pb0} and Proposition~\ref{prop:equiv_int_k_rho} also imply that a $[\mPb^p(\rho;\X)]^0$-s.n.t.\ space induces a $\mP^p(K_\rho;\X)$-characteristic kernel.\footnote{Note that the semimetric generated by a kernel $K_\rho$ equals that of the original semimetric $\rho$ \citep[Proposition 14]{sejdinovic13equivalence}: $\rho=\rho_{K_{\rho}}$.} \label{rem:6:2}

\item \tb{Example of a $\mP^2(K;\X)$-characteristic kernel:}
Any separable Hilbert space $\H$ is known to have $[\mPb^2(\norm{\cdot_{1}-\cdot_{2}}{\H};\H)]^0$-s.n.t.\ \citep[Theorem 3.16]{lyons13distance}, and hence by the last implication in the previous bullet (choosing $p=2$), $K_{\norm{\cdot}{\H}}$ is implied to be $\mP^2({\norm{\cdot_{1}-\cdot_{2}}{\H}};\H)=\mP^2(K_{\norm{\cdot}{\H}};\H)$-characteristic (the equality comes from Proposition~\ref{prop:equiv_int_k_rho} while choosing $p=2$).\label{rem:6:3}

\end{enumerate}
\end{remark}

\subsection{$L_p$-universality} \label{sec:Lp-universality}
Our next theorem relates the $L_p$-universality of a kernel $K$ to it being $\A^p(K;\X)$-i.s.p.d.

\begin{theorem}[Necessary and sufficient conditions for $L_p$-universality; see Appendix~\ref{subsec:proof_univ}]\label{th:univ}
    Let $(\X,\tau_\X)$ be a topological space, $p\in[1,\infty)$, and $\H_K$ a separable RKHS\textsuperscript{\ref{footnote:separable-Hk}} on $\X$ with kernel $K:\X^2\to\R$. Then, $K$ is $L_p$-universal iff $K$ is $\A^p(K;\X)$-i.s.p.d.
\end{theorem}

We comment on this result after showing a surprising phenomenon: $\A^p(K;\X)=\mPb(K;\X)$ under mild conditions, even when $K$ is unbounded, as captured in the following theorem.
\begin{theorem}[$\A^p(K;\X)$ equals $\mPb(K;\X)$; see Appendix~\ref{subsec:proof_Ap=Pb}] \label{th:Ap=Pb}
    Let $(\X,\tau_\X)$ be a topological space, $\H_K$ an RKHS with kernel $K:\X^2\to\R$, and $p\in[1,\infty)$. Then, $\A^p(K;\X) = \mPb(K;\X)$.
\end{theorem}

With both results stated, we make the following notes.
\begin{remark}~\label{rem:9}
    \begin{enumerate}[label=(\alph*)]

        \item \tb{Equivalence of $L_p$-universality and relationship to existing results:} Theorems~\ref{th:univ}--\ref{th:Ap=Pb} imply that an $L_p$-universal kernel is also $L_q$-universal, for any $p,q\in[1,\infty)$. This surprising result shows that the equivalence among $L_p$-universalities is not exclusive to real-valued $c_0$-kernels on locally compact second-countable spaces, as implied by \citet[Theorem~1]{carmeli10vector}.
        \label{rem:9.a}

        \item \tb{Example of a non-$L_{p}$-universal kernel:} An interesting consequence of Theorems~\ref{th:univ}--\ref{th:Ap=Pb} is that $K$ cannot be $L_p$-universal if it is not $\mPb(K;\X)$-i.s.p.d. In particular, this implies that a Stein kernel $K_{\P_0}$ on $\X$ \citep{hagrass26stein,cribeiro26minimax}, for the target $\P_0$, cannot be $L_p$-universal (for any $p\geq 1$) since $\norm{\mu_{K_{\P_0}}(\P_0)}{\H_{K_{\P_0}}}\coloneq0$  and $\P_0\in\mPb(K_{\P_0};\X)\setminus \{0\}$. \label{rem:9:d}

       \item  \tb{Example of an $L_p$-universal unbounded kernel:} The exponential kernel $K(\b x,\b y) \coloneq e^{a\fip{\b x,\b y}{}}$ ($a>0$) on $\R^d\times\R^d$ is $\mPb(K;\R^d)$-i.s.p.d.\ by Lemma~\ref{lem:exp_is_ispd}. Hence,
       by combining Theorem~\ref{th:univ} and Theorem~\ref{th:Ap=Pb}, it is $L_p$-universal for every $p\in [1,\infty)$.
       Using that in compact Hausdorff spaces $c_c$-universality\footnote{$K$ is $c_c$-universal iff $\overline\H_K = \mC(\X)$ w.r.t.\ the topology of compact convergence. \label{footnote:cc-univ}} is equivalent to $L_p$-universality \citep[Figure 1.2]{sriperumbudur11universality}, this result is known to hold on any compact subset of $\R^d$ \citep[Example 1]{steinwart01influence}; our results imply its $L_p$-universality on the non-compact space $\R^d$.\label{rem:9:e}

       \item \tb{Class of $L_p$-universal unbounded kernels:} By Lemma~\ref{lem:kc_ispd} (taking $p=1$, $\X=\R^d$), if $K$ is $\mP(K;\R^d)$-characteristic, then the shifted kernel $K_c = K + c$ ($c>0$) is $\mPb(K_c;\R^d)$-i.s.p.d.; hence, it is also $L_p$-universal by Theorems~\ref{th:univ}--\ref{th:Ap=Pb}. For instance, kernels taking the form
       \begin{align}
           K(\b x, \b y) &= \int_{\R^d} (1-e^{i\fip{\b s,\b x}{}})(1-e^{-i\fip{\b s,\b y}{}}) \d \Lambda(\b s) \\
           &\quad + \b x^\T\bm \Sigma \b y, \label{eq:modeste_kernel}
       \end{align}
       where $i = \sqrt{-1}$, $\bm \Sigma$ is a positive definite matrix and $\Lambda$ is a non-negative Borel measure are known \citep[Proposition~9]{modeste24characterization} to be $\mP(K;\R^d)$-characteristic under mild conditions; an example \citep[Example~4]{modeste24characterization} of a $\mP(K;\R^d)$-characteristic kernel of the shape \eqref{eq:modeste_kernel} is the fractional Brownian one: \begin{align}K(\b x,\b y)=\norm{\b x}{2}^{\alpha} + \norm{\b y}{2}^{\alpha} - \norm{\b x-\b y}{2}^\alpha\end{align} with $\alpha \in (0,2)$. Hence, the shifted fractional Brownian kernel $K_c(\b x, \b y) = K(\b x, \b y) + c$ (with $c>0$) is $L_p$-universal for any $p \in[1,\infty)$.

    \end{enumerate}
\end{remark}

Note that the case of $p=\infty$ is not handled by Theorem~\ref{th:univ}. To this end, we introduce the set
\begin{align}
\MoveEqLeft \mPb^\infty(\X) \coloneq \mPb^\infty(K;\X) \\
&\hspace{-0.6cm} \coloneq \{\F\in\M_b(\X):~\F\text{ has compact support}\}, \label{eq:Pb-infty}
\end{align}
and observe that $\mPb^\infty(\X)\subsetneq \mPb^p(K;\X)$ for any $p\in[1,\infty)$ as long as $\X$ is not compact and $K$ is continuous. Then, we have the following implication coming from \citep[Table 1]{simon-gabriel18kernel}.

\begin{proposition}[Equivalence to $c_c$-universality]\label{prop:cc-univ}
    Let $(\X,\tau_\X)$ be a second-countable Hausdorff space and $\H_K$ an RKHS with a continuous kernel $K:\X^2\to \R$. Then,  $K$ is $\mPb^\infty(\X)$-i.s.p.d.\ iff $K$ is $c_c$-universal.\textsuperscript{\ref{footnote:cc-univ}}
\end{proposition}

\subsection*{Acknowledgments}
FK is funded by the KiKIT (The Pilot Program for Core-Informatics at the KIT) of the Helmholtz Association.

\bibliography{BIB/old/collected,BIB/old/collected_Jose,BIB/old/publications_own}

\input{supplement_arxiv}

\end{document}

%% file: preamble_arxiv.tex
\usepackage[round]{natbib}
\usepackage{subfiles}

\usepackage{multirow}
\usepackage{hhline}
\usepackage{multicol}
\usepackage{mathtools}
\mathtoolsset{showonlyrefs}
\usepackage{amsmath,amssymb,amsthm}
\usepackage{bm}
\usepackage{dsfont}
\usepackage{enumitem}
\usepackage[final]{showlabels} %
\usepackage{xcolor}         %
\usepackage{booktabs} %
\usepackage{hyperref}
\usepackage[mathscr]{eucal}    %
\usepackage{subcaption} %
\usepackage{soul} %
\usepackage[all]{xy} %

\usepackage{tikz}
\usetikzlibrary{arrows} %
\usetikzlibrary{arrows.meta} 
\usetikzlibrary{cd} %
\usetikzlibrary{positioning} %
\usepackage{circuitikz}%
\usepackage{pifont}  %
\usepackage{centernot}

\newtheorem{theoremA}{Theorem}[section]
\newtheorem{lemmaA}[theoremA]{Lemma}
\newtheorem{corollaryA}[theoremA]{Corollary}
\newtheorem{propositionA}[theoremA]{Proposition}

\definecolor{Brown}{rgb}{0.64,0.16,0.16}
\definecolor{OliveGreen}{rgb}{0.1,0.4,0.1}
\definecolor{KITgreen}{RGB}{0,150,130}
\definecolor{KITblue}{RGB}{70,100,170}

\renewcommand{\b}{\mathbf}    %
\newcommand{\R}{\mathbb{R}}   %
\renewcommand{\P}{\mathbb P} %
\newcommand{\Q}{\mathbb Q} %
\newcommand{\F}{\mathbb F} %
\newcommand{\E}{\mathbb{E}} %
\newcommand{\T}{\top} %
\renewcommand{\d}{\mathrm{d}} %
\newcommand{\tb}{\textbf}    %
\newcommand{\fip}[2]{\left\langle{#1}\right\rangle_{#2}} %
\newcommand{\norm}[2]{\left\|{#1}\right\|_{#2}} %
\newcommand{\MMD}[1][K]{\operatorname{MMD}_{#1}}

\newcommand{\KSD}[1][K_{\P_0}]{\operatorname{KSD}_{#1}}
\renewcommand{\H}{\mathscr{H}} %
\newcommand{\X}{\mathcal{X}} %
\newcommand{\N}{\mathbb{N}} %
\newcommand{\Np}{\mathbb{Z}_{+}}%
\newcommand{\Rnn}{\R_{\ge 0}} %
\renewcommand{\d}{\mathrm{d}} %
\newcommand{\M}{\mathcal{M}}
\newcommand{\jover}[2]{\overset{\text{(#1)}}{#2}} %

\newcommand{\mP}{\mathcal{P}_{1}} %
\newcommand{\mPb}{\mathcal{P}_{b}} %
\newcommand{\mF}{\mathscr{F}} %
\newcommand{\mC}{\mathcal{C}}
\newcommand{\mL}{\mathcal{L}}
\newcommand{\mB}{\mathcal{B}}

\DeclareMathOperator{\id}{id}

\newtheorem{theorem}{Theorem}
\newtheorem{lemma}[theorem]{Lemma} 
\newtheorem{proposition}[theorem]{Proposition} 
\newtheorem{remark}[theorem]{Remark}

\newcommand{\CurrentTitle}{Unbounded Characteristic and Universal Kernels}

\definecolor{gray1}{rgb}{0.97,0.97,0.97}
\definecolor{gray2}{rgb}{0.9,0.9,0.9}
\definecolor{gray3}{rgb}{0.75,0.75,0.75}
\definecolor{gray4}{rgb}{0.6,0.6,0.6}
\definecolor{gray5}{rgb}{0.45,0.45,0.45}

\DeclareMathOperator{\Abs}{Abs} %
\newcommand{\A}{\mathcal{A}} %
\newcommand{\nK}{\norm{K(\cdot,x)}{\H_K}} %

%% file: supplement_arxiv.tex
\setcounter{equation}{0}
\renewcommand{\theequation}{\thesection.\arabic{equation}} %

\appendix
\onecolumn

\aistatstitle{\CurrentTitle: \\
Supplementary Materials}

\begin{table}[h]%
    \caption{Summary of own and external results in the supplement.}
    \label{table:results-in-the-supp}
    \begin{center}
    \begin{tabular}{ l ll}
    \toprule 
    Statement & Type & Content \\
    \midrule
    Lemma~\ref{lem:f_unbounded_not_q_int} & own & unbounded $f\ge0$ admits a $\P$ with $f\in L_1(\X,\P)$ but $f\not\in L_q(\X,\P)$ for $q>1$\\
    Lemma~\ref{lem:decomp_F} & own & a measure in {$[\mPb^p(K;\X)]^0$} has a decomposition in $\mP^p(K;\X)$  \\
    Lemma~\ref{lem:decomp_F_rho} & own & a measure in {$[\mPb^p(\rho;\X)]^0$} has a decomposition in $\mP^p(\rho;\X)$  \\
    Lemma~\ref{lem:inj_barycenter} & own & injectivity of the barycenter implies s.n.t.\\
    Lemma~\ref{lem:lp_isom_abs} & own & $\Abs_{p}(\P)$ and $L_p(\X,\P)$ are isometrically isomorphic\\
    Corollary~\ref{lem:abs_holds_ii-iii} & own & $\Abs_{p'}(\P)$ is a group\\
    Lemma~\ref{lem:abs_p_subset_mpb} & own & a property of the space $\Abs_{p'}(\P)$ for measures in $\mP^p(K;\X)$\\
    Lemma~\ref{lem:Pbp-vs-Ap} & own & $\mPb^p(K;\X) \subseteq \A^p(K;\X)$\\
    Lemma~\ref{lem:exp_is_ispd} & own & the exponential kernel is $\mPb(K;\R^d)$-i.s.p.d.\\
    Lemma~\ref{lem:kc_ispd} & own & shifted $\mP^p(K;\X)$-characteristic kernels are $\mPb^p(K;\X)$-i.s.p.d. \\
    Proposition~\ref{prop:equiv_int_k_rho} & external & integrability of kernels and semimetrics\\
    Theorem~\ref{th:Fubini} & external & Fubini-Tonelli\\
    Theorem~\ref{th:ingo_4.26}  & external & properties of $S_K$ \\
    Lemma~\ref{lem:hoelder}  & external & a generalized Hölder's inequality\\
    \bottomrule
    \end{tabular}
    \end{center}
\end{table}

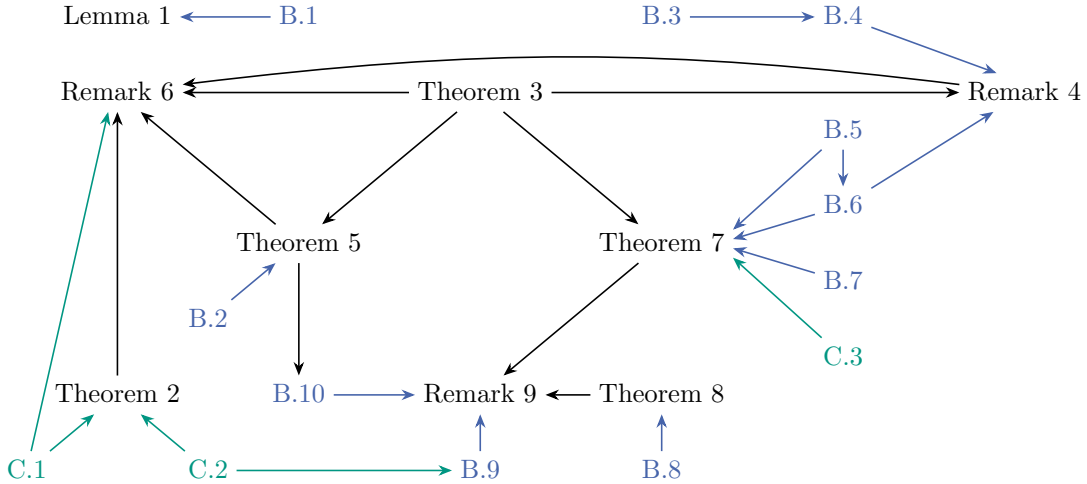
\begin{figure}[h]
    \centering
    \begin{tikzpicture}[
    x=2.4cm, y=1cm,
    every node/.style={inner sep=3pt, text height=1.6ex, text depth=0.3ex},
    res/.style={text=black},
    bres/.style={text=KITblue},
    cres/.style={text=KITgreen},
    dep/.style={-{Stealth[length=1.8mm, width=1.4mm]}, semithick},
    bdep/.style={dep, draw=KITblue},
    cdep/.style={dep, draw=KITgreen},
  ]
  \def\tip{3.5pt} %

  \node[res] (L1) at ( -2,  1) {Lemma~\ref{lem:monotonicity_mpb}};
  \node[res] (R6) at (-2,  0) {Remark~\ref{rem:6}};
  \node[res] (T3) at ( 0,  0) {Theorem~\ref{th:ispd0_iff_char}};
  \node[res] (R4) at ( 3,  0) {Remark~\ref{rem:4}};
  \node[res] (T5) at (-1, -2) {Theorem~\ref{th:char_p1_iff_pb0}};
  \node[res] (T7) at ( 1, -2) {Theorem~\ref{th:univ}};
  \node[res] (T2) at (-2, -4) {Theorem~\ref{th:snt_iff_ispd}};
  \node[res] (R9) at ( 0, -4) {Remark~\ref{rem:9}};
  \node[res] (T8) at ( 1, -4) {Theorem~\ref{th:Ap=Pb}};

  \node[bres] (B3)  at (1,  1)   {\ref{lem:decomp_F_rho}};
  \node[bres] (B4)  at (2, 1)  {\ref{lem:inj_barycenter}};
  \node[bres] (B1)  at ( -1,  1)   {\ref{lem:f_unbounded_not_q_int}};
  \node[bres] (B6)  at ( 2, -0.5) {\ref{lem:lp_isom_abs}};
  \node[bres] (B7)  at ( 2, -1.5) {\ref{lem:abs_holds_ii-iii}};
  \node[bres] (B8)  at ( 2, -2.5) {\ref{lem:abs_p_subset_mpb}};
  \node[bres] (B2)  at (-1.5, -3)   {\ref{lem:decomp_F}};
  \node[bres] (B10) at ( 0, -5)   {\ref{lem:exp_is_ispd}};
  \node[bres] (B9)  at ( 1, -5)   {\ref{lem:Pbp-vs-Ap}};
  \node[bres] (Bnew) at (-1, -4) {\ref{lem:kc_ispd}};

  \node[cres] (C3) at ( 2,   -3.5) {\ref{th:ingo_4.26}};
  \node[cres] (C1) at (-2.5, -5)   {\ref{prop:equiv_int_k_rho}};
  \node[cres] (C2) at (-1.5, -5)   {\ref{th:Fubini}};

  \draw[dep] (T3) -- (R6);
  \draw[dep] (T3) -- (R4);
  \draw[dep] (T3) -- (T5);
  \draw[dep] (T3) -- (T7);
  \draw[dep] (T5) -- (R6);
  \draw[dep] (T2) -- (R6);
  \draw[dep] (T7) -- (R9);
  \draw[dep] (T8) -- (R9);
  \draw[dep] (T5) -- (Bnew);
  \draw[dep] (R4) to [bend right=7] (R6);

  \draw[bdep] (B1)  -- (L1);
  \draw[bdep] (B3)  -- (B4);
  \draw[bdep] (B2)  -- (T5);
  \draw[bdep] (B6)  -- (B7);
  \draw[bdep] (B7)  -- (R4);
  \draw[bdep] (B10) -- (R9);
  \draw[bdep] (B9)  -- (T8);
  \draw[bdep] (Bnew) -- (R9);
  \draw[bdep] (B4) -- ([xshift=-3*\tip]R4.north);
  \draw[bdep] (B6) -- ([yshift= 1.5*\tip]T7.east);
  \draw[bdep] (B7) -- ([yshift= 0.5*\tip]T7.east);
  \draw[bdep] (B8) -- ([yshift=-0.5*\tip]T7.east);

  \draw[cdep] (C3) -- ([yshift=-1.5*\tip]T7.east);
  \draw[cdep] (C1) -- (T2);
  \draw[cdep] (C1) -- ([xshift= -1*\tip]R6.south);
  \draw[cdep] (C2) -- (T2);
  \draw[cdep] (C2) -- (B10);

\end{tikzpicture}
    \caption{Summary of the dependencies of our results and remarks. $R_1 \leftarrow R_2$ means that ``result $R_1$ depends on $R_2$''. Results and remarks from the main text are written in black, our auxiliary results are in blue, and external results are indicated by green.}
    \label{fig:results_relations}
\end{figure}

For the readers' convenience, a summary of the results in the supplement (both own and external) is provided in Table~\ref{table:results-in-the-supp}, the relations of all the results (both in the main and in the supplement) are visualized in Figure~\ref{fig:results_relations}.

\section{PROOFS} \label{sec:proofs}
This section is dedicated to the proofs of our statements in the main text. The proof of Lemma~\ref{lem:monotonicity_mpb} is in Appendix~\ref{subsec:proof_monotonicity}, that of Theorem~\ref{th:snt_iff_ispd} is in Appendix~\ref{subsec:proof_snt}, and that of Theorem~\ref{th:ispd0_iff_char} is in Appendix~\ref{subsec:proof_ispd0_iff_char}. Theorem~\ref{th:char_p1_iff_pb0} is proved in Appendix~\ref{subsec:proof_char_p1_iff_pb0}, Theorem~\ref{th:univ} is shown in Appendix~\ref{subsec:proof_univ}, and the proof of Theorem~\ref{th:Ap=Pb} can be found in Appendix~\ref{subsec:proof_Ap=Pb}.

\subsection{Proof of Lemma~\ref{lem:monotonicity_mpb}}\label{subsec:proof_monotonicity}
We prove the two statements separately.
  \paragraph{Part i).}
  As $\mPb^q(K;\X)\subseteq \mPb^{p}(K;\X)$ and $\mP^q(K;\X)\subseteq \mP^{p}(K;\X)$ for $q\ge p$ (see Section~\ref{sec:notations}), it is sufficient to show that $\mPb^{p}(K;\X)\subseteq \mPb^q(K;\X)$ and $\mP^{p}(K;\X)\subseteq \mP^q(K;\X)$ for $q\ge p$ and bounded $K$. Let $\F\in \mPb^{p}(K;\X)$ be arbitrary. Then
  \begin{align}
    \int_\X \norm{K(\cdot,x)}{\H_K}^q \d \F(x) &\jover a \le  \int_\X \norm{K(\cdot,x)}{\H_K}^q \d |\F|(x) \jover b\le \int_\X \sup_{x\in\X} \norm{K(\cdot,x)}{\H_K}^q \d |\F|(x) 
      \\&\jover c= \sup_{x\in\X}\norm{K(\cdot,x)}{\H_K}^q|\F|(\X) 
      \jover d < \infty, \notag
  \end{align}
  where (a) follows as the integrand is non-negative, (b) comes from upper bounding the integrand and the monotonicity of integration, (c) is by the  linearity of integration, and (d) holds as the assumed boundedness of $K$ and $\F\in\mathcal \M_b(\X)$ implies $\sup_{x\in\X}\norm{K(\cdot,x)}{\H_K}^q<\infty$ and $|\F|(\X)<\infty$, respectively. 
  The obtained inequality shows that $\F\in \mPb^q(K;\X)$, as was to be proved. As $\mP^p(K;\X)\subseteq \mPb^p(K;\X)$ and $\F$ was taken arbitrarily, the result also holds for  probability measures. Further, as $K$ is bounded, $\mPb(K;\X) = \M_b(\X)$ and $\mP(K;\X)=\M_1^+(\X)$ (as the kernel is integrable w.r.t.\ any finite signed measure).

 \paragraph{Part ii).} As stated in Section~\ref{sec:notations}, it holds that $\mPb^q(K;\X)\subseteq \mPb^{p}(K;\X)$ and $\mP^q(K;\X)\subseteq \mP^{p}(K;\X)$ for $q\ge p$; it remains to show that the inclusion is strict when $K$ is unbounded.
 Let $r=q/p>1$, implying that $pr = q$. As $K$ is unbounded, $\norm{K(\cdot,x)}{\H_K}^p:\X\to\Rnn$ is unbounded. Hence, by Lemma~\ref{lem:f_unbounded_not_q_int}, there exists a
 probability measure $\P\coloneq\P(r)\in\M_1^+(\X)$ such that 
    \begin{align}
        &\int_\X \norm{K(\cdot,x)}{\H_K}^p \d \P(x) <\infty, \text{ while } \label{eq:relationship:finiteness}\\
        & \int_\X \norm{K(\cdot,x)}{\H_K}^{pr} \d \P(x)   = \int_\X \norm{K(\cdot,x)}{\H_K}^{q} \d \P(x)=\infty.\label{eq:relationship:divergence}
      \end{align}
    Then, $\P\overset{\eqref{eq:relationship:finiteness}}{\in}\mP^p(K;\X)$ but $\P\overset{\eqref{eq:relationship:divergence}}{\notin}\mP^{q}(K;\X)$, proving the claim for probability measures. At the same time, as $\P\in\mP^{p}(K;\X)=\mPb^p(K;\X)\cap\M_1^+(\X)$ and $\mP^{q}(K;\X)=\mPb^q(K;\X)\cap\M_1^+(\X)$, we proved the claim for finite signed measures. 

\subsection{Proof of Theorem~\ref{th:snt_iff_ispd}}\label{subsec:proof_snt}
   Assume that $(\X,\rho_K)$ has $\mF$-s.n.t., and let $\F\in \mF\setminus \{0\}$. By Proposition~\ref{prop:equiv_int_k_rho} (valid as $p=2\in\N$), $\mPb^2(K;\X)=\mPb^2(\rho_K;\X)$; hence the double integral in \eqref{eq:snt} is well-defined for $\F$, and 
    \begin{align}
        0&>\int_{\X}\int_{\X}\rho_K(x,y) \d \F(x)\d \F(y)
        \jover a=\!\!\! \int_{\X}\!\int_{\X}\!\! K(x,x) + K(y,y) -2K(x,y) \d\F(x)\d\F(y)\\
        &\jover b=\!\!\! \int_{\X}\!\int_{\X}\!\! K(x,x)\d\F(x)\d\F(y) +  \int_{\X}\!\int_{\X}\!\! K(y,y) \d\F(x)\d\F(y) -2\int_{\X}\!\int_{\X}\!\! K(x,y) \d\F(x)\d\F(y)\\
        &\jover c=\!\!\! \int_{\X}\!\int_{\X}\!\! K(x,x)\d\F(y)\d\F(x) +  \int_{\X}\!\int_{\X}\!\! K(y,y) \d\F(x) \d\F(y) -2\int_{\X}\!\int_{\X}\!\! K(x,y) \d\F(x)\d\F(y)\\
        &\jover d=\! \underbrace{\F(\X)}_{=0}\int_{\X}\!\! K(x,x)\d\F(x) + \underbrace{\F(\X)}_{=0} \int_{\X}\!\! K(y,y) \d\F(y) -2\int_{\X}\!\int_{\X}\!\! K(x,y) \d\F(x)\d\F(y)
        \\&\jover e= -2\int_{\X}\!\int_{\X}\!\! K(x,y) \d\F(x)\d\F(y), \label{eq:K:F-ispd}
    \end{align}
    where (a) comes from the definition of kernel-induced semimetric, (b) is by the linearity of integration, (c) follows from applying Tonelli's theorem (recalled in Theorem~\ref{th:Fubini}) to the first addend, and (d) is implied by the homogeneity of the integral. Lastly, (e) comes by the finiteness of both integrals along the diagonal as $\F \in \mPb^2(K;\X)$.
    This implies that $\int_\X\int_\X K(x,y)\d \F(x)\d\F(y) > 0$, as was to be shown.

    Considering the opposite direction (in other words, assuming that $K$ is $\mF$-i.s.p.d.), (i) by the semimetric property of $\rho_K$,  $K$ is non-degenerate which itself implies that  $\rho_K$ is  of negative type,\textsuperscript{\ref{footnote:rho-K-semi-metric:sufficient-condition}} and (ii)  for any $\F \in \mF\setminus\{0\}$, $\int_{\X}\!\int_{\X}\! K(x,y) \d\F(x)\d\F(y)>0$ by the $\mF$-i.s.p.d.\ property of $K$ which yields that $\int_{\X}\int_{\X}\rho_K(x,y) \d \F(x)\d \F(y)<0$ by \eqref{eq:K:F-ispd}; these 2 properties mean that $(\X,\rho_K)$ has $\mF$-s.n.t. \qed

\subsection{Proof of Theorem~\ref{th:ispd0_iff_char}}\label{subsec:proof_ispd0_iff_char}
If $\mF=\{0\}$, both implications hold; in the following we assume that $\mF\neq\{0\}$.
\paragraph{$(\!\!\implies\!\!)$} 
Assume that $K$ is $\mF$-characteristic. Then, for all $\F_1\ne \F_2 \in \mF$,
\begin{align}
0 < \norm{\mu_K(\F_1)-\mu_K(\F_2)}{\H_K} \jover{a}{=} \norm{\mu_K(\F_1-\F_2)}{\H_K}, \label{eq:K-Fchar}
\end{align}
where (a) holds by the linearity of Bochner integration w.r.t.\ the measure. For all $\F\in\mF\setminus\{0\}$,  \eqref{eq:K-Fchar} with a substitution $\F_1 \leftarrow \F$, $\F_2 \leftarrow 0 \in \mF$  implies that
\begin{align}
    0&<\norm{\mu_K(\F)}{\H_K}^2  \overset{\eqref{eq:def:ispd}} = \int_\X\int_\X K(x,y)\d\F(x)\d\F(y)
\end{align}
proving that $K$ is $\mF$-i.s.p.d.

\paragraph{$(\!\!\impliedby\!\!)$} 
 Assume that $K$ is $\mF$-i.s.p.d.\ and $\mF$ is closed w.r.t.\ negation and addition.
 Let $\F_1\ne\F_2\in\mF$. Then, $\F\coloneq\F_1-\F_2\in\mF\setminus\{0\}$ ($\F\in \mF$ is implied by the  previous 2 closeness properties of $\mF$; $\F \ne 0$ holds as $\F_1 \ne \F_2$), and
 \begin{align}
    0 & \jover{a}{<}   \int_\X\int_\X K(x,y)\d \F(x)\d \F(y)\overset{\eqref{eq:def:ispd}}{=} \norm{\mu_K(\F)}{\H_K}^2
    \jover {b}{=}\norm{\mu_K(\F_1)-\mu_K(\F_2)}{\H_K}^{2} \jover {c}{\implies}
    \mu_K(\F_1)  \neq\mu_K(\F_2),
 \end{align}
 where
 (a) follows from the $\mF$-i.s.p.d.\ property  of $K$ (using that $\F\ne 0$),
 (b) holds by the linearity of the Bochner integration w.r.t.\ the measure, 
and (c) comes from the fact that $\norm{\cdot}{\H_K}$ is a norm. Thus, $\mu_K$ is injective on $\mF$, in other words  $K$ is $\mF$-characteristic. \qed

 \subsection{Proof of Theorem~\ref{th:char_p1_iff_pb0}} \label{subsec:proof_char_p1_iff_pb0}
 Let $p\in [1,\infty)$ be fixed. We start with an equality used throughout the proof. Indeed, suppose that $\F\in[\mPb^p(K;\X)]^0\setminus\{0\}$ and $\F = M(\P-\Q)$  with distinct $\P,\Q \in\mP^p(K;\X)$ and $M\in(0,\infty)$ (guaranteed by Lemma~\ref{lem:decomp_F}). Then
 \begin{align}
     \MoveEqLeft \norm{\mu_K(\F)}{\H_K} \jover{a}{=} \norm{M\mu_K(\P)-M\mu_K(\Q)}{\H_K} \jover{b}{=}  M\norm{\mu_K(\P)-\mu_K(\Q)}{\H_K},%
     \label{eq:mu_K(F)_is_linear}
 \end{align}
 where (a) comes by the linearity of the Bochner integration w.r.t.\ the measure, (b) is by the absolute homogeneity of the norm combined with $M>0$.%

\paragraph{$(\!\!\implies\!\!)$}
Let $K$ be $\mP^p(K;\X)$-characteristic; our goal is to show that it is $[\mPb^p(K;\X)]^0$-characteristic, in other words, that for any $\F_1 \ne \F_2$ in $[\mPb^p(K;\X)]^0$, $\mu_K(\F_1)\ne \mu_K(\F_2)$. Taking such $\F_1\ne \F_2$, for $\F\coloneq \F_1-\F_2\in  [\mPb^p(K;\X)]^0 \setminus \{0\}$, we have that
$\norm{\mu_K(\F)}{\H_K}>0$ by \eqref{eq:mu_K(F)_is_linear}, the imposed $\mP^p(K;\X)$-characteristic property of $K$, the fact that $\norm{\cdot}{\H_K}$ is a norm, and $M>0$. Noticing that any non-zero $\F \in [\mPb^p(K;\X)]^0$ can be written as\footnote{Indeed, by choosing $\F_1 \leftarrow \F\ne 0$, $\F_2 \leftarrow 0 \in  [\mPb^p(K;\X)]^0$.} 
$\F=\F_1-\F_2$ (with $\F_1 \neq \F_2$ in $[\mPb^p(K;\X)]^0$) the obtained property $0<\norm{\mu_K(\F)}{\H_K}^2\stackrel{\eqref{eq:def:ispd}}{=} \int_\X\int_\X K(x,x')\d\F(x)\d\F(x')$ shows that  $K$ is $[\mPb^p(K;\X)]^0$-i.s.p.d., which yields that $K$ is $[\mPb^p(K;\X)]^0$-characteristic  by Theorem~\ref{th:ispd0_iff_char} [see Remark~\ref{rem:4}\ref{remark4:(a)}]. 

\paragraph{$(\!\!\impliedby\!\!)$} 
Assume that $K$ is $[\mPb^p(K;\X)]^0$-characteristic; our goal is to prove that $K$ is $\mP^p(K;\X)$-characteristic, in other words, that for any $\P\ne \Q \in \mP^p(K;\X)$, $\mu_K(\P) \ne \mu_K(\Q)$.  Considering such arbitrary $\P\ne \Q$, define $\F= \P - \Q \in  [\mPb^p(K;\X)]^0\setminus\{0\}$. The imposed $[\mPb^p(K;\X)]^0$-characteristic property of $K$, implies that $K$ is $[\mPb^p(K;\X)]^0$-i.s.p.d.\ by Theorem~\ref{th:ispd0_iff_char}, and hence by $\F\neq0$ that 
\begin{align}0 & < \int_\X\int_\X K(x,x')\d\F(x)\d\F(x') \stackrel{\eqref{eq:def:ispd}}{=} \norm{\mu_K(\F)}{\H_K}^2 \overset{\eqref{eq:mu_K(F)_is_linear}}{=} \norm{\mu_K(\P)-\mu_K(\Q)}{\H_K}^2,
\end{align}
with $M=1$ in \eqref{eq:mu_K(F)_is_linear}.
Hence, $\mu_K(\P)\neq\mu_K(\Q)$ since $\norm{\cdot}{\H_K}$ is a norm, implying the injectivity of $\mu_K$ on  $\mP^p(K;\X)$ which we wanted to prove. \qed

\subsection{Proof of Theorem~\ref{th:univ}}\label{subsec:proof_univ}
We begin by introducing some pre-requisites.
Fix $p\in[1,\infty)$ and $\P\in\mP^p(K;\X)\subseteq \mP(K;\X)$. 
Further, consider the map $\imath_\P:  L_{p'}(\X,\P)\to\Abs_{p'}(\P)$, $[f]_\sim\mapsto \F_{\P,f}$ as in Lemma~\ref{lem:lp_isom_abs}, where $\F_{\P,f}$ is such that $\frac{\d \F_{\P,f}}{\d \P}=f$.
Recall from Lemma~\ref{lem:abs_p_subset_mpb} that $\F_{\P,f}\in\Abs_{p'}(\P)\subseteq\mPb(K;\X)$.

Lastly, note that the mean embedding $\mu_K$, $\imath_\P$, and $S_K$ form the following diagram
\begin{equation}\label{eq:sk_diag}
\begin{tikzcd}
{L_{p'}(\X,\P)} \arrow[r, "S_K"] \arrow[rd, "\imath_\P"'] & \H_K \\ 
& \Abs_{p'}(\P) \arrow[u, "\mu_K|_{\Abs_{p'}(\P)}"'] 
\end{tikzcd}
\end{equation}
where $\mu_K|_{\Abs_{p'}(\P)}$ denotes the restriction of $\mu_K$ to $\Abs_{p'}(\P)\subseteq \mPb(K;\X)$. Our goal is to prove the injectivity of various parts of this diagram. On one hand this will help us to prove the desired denseness (via Theorem~\ref{th:ingo_4.26}), and on the other it will allow us to prove the $\A^p(K;\X)$-i.s.p.d.\ property of $K$ (via Theorem~\ref{th:ispd0_iff_char}).   

\paragraph{$(\!\!\implies\!\!)$} 
Recall that by definition $\A^p(K;\X)=\cup_{\P\in\mP^p(K;\X)}\Abs_{p'}(\P)$, hence it is sufficient to prove that $K$ is $\Abs_{p'}(\P)$-i.s.p.d.\ for any $\P\in \mP^p(K;\X)$; this is what we detail next.
Indeed, as $K$ is $L_p$-universal, for any $\P\in\mP^p(K;\X)$, 
\begin{align}
    \overline{\H}_K = L_p(\X,\P) &\jover a\implies S_K \text{ is injective } 
                                  \jover b\implies \mu_K|_{\Abs_{p'}(\P)} \text{ is injective }
                                  \jover {def.}\iff \mu_K \text{ is injective  on } \Abs_{p'}(\P)\\
                                 &\jover{def.}\iff K \text{ is $\Abs_{p'}(\P)$-characteristic}
                                  \jover c\implies K \text{ is $\Abs_{p'}(\P)$-i.s.p.d.},
\end{align}
where (a) comes from Theorem~\ref{th:ingo_4.26}, (b) holds as $\mu_K|_{\Abs_{p'}(\P)}=S_K\circ \imath_\P^{-1}$ is the composition of injective functions (see Lemma~\ref{lem:lp_isom_abs}),
and (c) comes from Theorem~\ref{th:ispd0_iff_char}  [valid by Remark~\ref{rem:4}\ref{remark4:(a)}]. 

\paragraph{$(\!\!\impliedby\!\!)$} Since $K$ is $\A^p(K;\X)$-i.s.p.d., it is also $\Abs_{p'}(\P)$-i.s.p.d.\  for any $\P\in\mP^p(K;\X)$ as $\A^p(K;\X)=\cup_{\P\in\mP^p(K;\X)}\Abs_{p'}(\P)$. Further, since $\Abs_{p'}(\P)$ satisfies the conditions of Theorem~\ref{th:ispd0_iff_char} (see Corollary~\ref{lem:abs_holds_ii-iii}), $K$ is $\Abs_{p'}(\P)$-characteristic, in other words $\mu_K$ is injective on $\Abs_{p'}(\P)$.
Therefore, $S_K = \mu_K|_{\Abs_{p'}} \circ \imath_\P$ as the composition of 2 injective maps (by~\eqref{eq:sk_diag} and Lemma~\ref{lem:lp_isom_abs}) is injective for all $\P\in\mP^p(K;\X)$, and thus, $\overline{\H}_K=L_p(\X,\P)$ by Theorem~\ref{th:ingo_4.26} for all $\P\in\mP^p(K;\X)$. \qed

\subsection{Proof of Theorem~\ref{th:Ap=Pb}}\label{subsec:proof_Ap=Pb}
    Let $p\in[1,\infty)$ be fixed. Note that $\mPb(K;\X)\supseteq \A^p(K;\X)$ holds by Lemma~\ref{lem:abs_p_subset_mpb}, hence it is sufficient to prove that also $\mPb(K;\X)\subseteq \A^p(K;\X)$.

    For $p=1$, Lemma~\ref{lem:Pbp-vs-Ap}   yields the result, therefore we assume in the following that $p\in(1,\infty)$. 
       
    In order to prove that in this case $\mPb(K;\X)\subseteq \A^p(K;\X)$, we consider an arbitrary $\F\in\mPb(K;\X)\setminus\{0\}$\footnote{The case $\F=0$ is immediate, as it is absolutely continuous to any probability measure $\P$ and it has a bounded Radon-Nikodym derivative $\frac{\d 0}{\d \P} = 0$.} and construct a non-negative $\nu\in\M_b(\X)$ such that (i) $\nu\in\mPb^p(K;\X)$, (ii) $\F\ll\nu$, and (iii) $\frac{\d \F}{\d \nu}\in L_{p'}(\X,\nu)$. Then, the re-scaled measure $\P \coloneq \nu/\nu(\X)$ satisfies $\P\in\mP^p(K;\X)$, $\F\ll\P$, and $\frac{\d \F}{\d \P}\in L_{p'}(\X,\P)$, which implies that $\F \in \Abs_{p'}(\P) \jover {Def.}{\subseteq} \A^p(K;\X)$. The details are as follows.

    First, to construct such $\nu\in\M_b(\X)$, consider the $\X \to \R$ function 
    \begin{align}
    f_K(x) \coloneq \big(1+\nK^p\big)^{-1/p'}. \label{eq:f_K}
    \end{align}
    $f_K$ is well-defined and strictly positive as $1+\nK^p> 0$ for every $x\in\X$, and it is measurable as it is the composition of measurable functions. Furthermore, $f_K\in L_{1}(\X,|\F|)$:
    \begin{align}
        \int_\X f_K(x) \d|\F|(x) \overset{\eqref{eq:f_K}}{=} \int_\X \frac{1}{\Big(1+ \nK^p\Big)^{1/p'}} \d|\F|(x) 
        \jover a\leq  \int_\X 1^{-1/p'}\d|\F|(x) = \int_\X 1 \d|\F|(x) \jover b< \infty,
    \end{align}
    where (a) is implied by the denominator being lower bounded by $1$, and (b) holds as $|\F|$ is a finite measure. Next, let us define the non-negative measure 
    \begin{align}
    \nu (A) &= \int_A f_K(x) \d |\F|(x), \quad (A\in\mB_\X),
    \end{align}
    where by definition $\nu\ll|\F|$ and $\frac{\d \nu}{\d|\F|} = f_K$; note that $\nu(\X)>0$ as $f_K$ is strictly positive and $|\F|\ne0$ (as $\F\ne0$). We will show that (i) $\nu\in\mPb^p(K;\X)$, (ii) $\F\ll \nu$, and (iii) $f_K\in L_{p'}(\X,\nu)$ after introducing two helpful inequalities. First, for $x\in \X$,
     \begin{align}
       \MoveEqLeft \nK^p f_K(x) \overset{\eqref{eq:f_K}}{=} \frac{\nK^p}{\big(1+\nK^p\big)^{1/p'}} 
        \jover{a}{\leq} \frac{1+\nK^p}{\big(1+\nK^p\big)^{1/p'}} 
        = \big(1+\nK^p\big)^{1-\frac{1}{p'}} 
        \\&\jover b = \big(1+\nK^p\big)^{1/p}, \label{eq:nkp.f_k}
    \end{align}
    and next, for $x\in \X$,
    \begin{align}
        \big(1+\nK^p\big)^{1/p}  \jover c\leq  1+\nK, \label{eq:1+nk}
    \end{align}
    where (a) follows from increasing the numerator [$\nK^p\leq 1 + \nK^p$] of the non-negative fraction, (b) holds as $p'$ is the conjugate exponent of $p$, and (c) holds by the inequality $(1+x^p)^{1/p}\leq 1+ x$ for $x\ge0$ and $p\ge1$. To see this, notice that the inequality is equivalent to $1+x^p\le (1+x)^p$. Letting $f(x) = (1+x)^p-1-x^p$ on $[0,\infty)$, noting that $f(0) = 0$, and $f'(x) = p\left[(1+x)^{p-1}-x^{p-1}\right]\ge 0$ for $x\ge 0$ and $p\ge1$ shows the claim.

    \paragraph{Part (i) [$\nu \in\mPb^p(K;\X)$].} We continue by showing that $\nu \in\mPb^p(K;\X)$:
    \begin{align}
       \MoveEqLeft \int_\X \nK^p \d\nu(x) \jover{a}{=} \int_\X \nK^p f_K(x) \d|\F|(x) \jover{b}{\leq} \int_\X \Big(1+\nK^p\Big)^{1/p} \d |\F|(x) \\
        &\jover{c}{\leq} \int_\X 1+\nK \d |\F|(x) \jover {d}{=} \underbrace{|\F|(\X)}_{<\infty} +\int_\X \nK \d|\F|(x) \jover{e}{<} \infty,
    \end{align}
    where (a) follows from the change of measure property of Radon-Nikodym derivatives and $\frac{\d \nu}{\d |\F|} = f_K$, (b) comes from \eqref{eq:nkp.f_k} and the monotonicity of integration, (c) holds by \eqref{eq:1+nk} and the monotonicity of integration, (d) is by the linearity of integration, and (e) holds by the assumption that $\F\in\mPb(K;\X)$.

    \paragraph{Part (ii) [$\F\ll\nu$].}
    Next, to show that $\F\ll\nu$, we will show that $|\F|\ll \nu$ (the transitivity of the absolute continuity of measures and the fact that $\F\ll|\F|$ implies that $\F\ll|\F|\ll\nu$; \citealt[Proposition 3.9.b]{folland99real}).
    For any $A\in\mB(\X)$\begin{align}
        |\F|(A) = \int_A 1 \d |\F|(x) \jover a = \int_A \frac{f_K(x)}{f_K(x)} \d|\F|(x) \jover b = \int_A \frac{1}{f_K(x)} \d \nu(x), \label{eq:inv_rn}
    \end{align}
    where (a) comes by multiplying and dividing by $f_K(x)$ [valid as $f_K(x)>0$ for all $x\in\X$], and (b) holds by the change of measure property of Radon-Nikodym derivatives and $\frac{\d \nu}{\d |\F|} = f_K$. 
    
    \paragraph{Part (iii) [$\frac{\d \F}{\d\nu}\in L_{p'}(\X,\nu)$].} 
    Last, we  show that $\frac{\d \F}{\d\nu} \jover{a}{=} \frac{\d \F}{\d |\F|} \frac{\d |\F|}{\d \nu} 
    \jover{b}{=} \frac{\d \F}{\d |\F|} \frac{1}{f_K} \in L_{p'}(\X,\nu)$, where (a) holds by the chain rule of the Radon-Nikodym derivative \citep[Proposition 3.9.b]{folland99real}, (b) follows from \eqref{eq:inv_rn}. Indeed,
    \begin{align}
        \int_\X \left|\frac{\d \F}{\d |\F|}(x) \frac{1}{f_K(x)} \right|^{p'}\d\nu(x) &=  \int_\X \left|\frac{\d \F}{\d |\F|}(x)\right|^{p'} \left|\frac{1}{f_K(x)}\right|^{p'}\d\nu(x) 
        \jover{a}{=} \int_\X 1^{p'} \left|\frac{1}{f_K(x)}\right|^{p'}\d\nu(x) 
        \\&\jover{b}{=} \int_\X \left[f_K(x)\right]^{-p'}\d\nu(x) \jover {c}{=} \int_\X f_K(x) \left[f_K(x)\right]^{-p'}(x) \d |\F|(x) 
        \\& = \int_\X \left[f_K(x)\right]^{1-p'} \d |\F|(x) \overset{\eqref{eq:f_K}}{=} \int_\X \Big[\big(1+\norm{K(\cdot,x)}{\H_K}^p\big)^{-1/p'}\Big]^{1-p'}\d|\F|(x)
        \\&\jover{d}{=} \int_\X \big( 1+\nK^p \big)^{1/p} \d|\F|(x) \jover{e}{\leq}\int_\X 1+\nK \d|\F|(x) 
        \\& \jover{f}{=} \underbrace{|\F|(\X)}_{<\infty} + \int_\X \nK \d|\F|(x) \jover{g}{<} \infty,
    \end{align}
    where (a) comes by $\left|\frac{\d \F}{\d|\F|}\right| = 1$ \citep[Exercise 13.7]{bass11real}, (b) holds by $f_K>0$, (c) follows from the change of measure property of Radon-Nikodym derivatives and $\frac{\d \nu}{\d|\F|}=f_K$, and (d) is by the fact that $-\frac{1-p'}{p'}=\frac{p'-1}{p'}=1-\frac{1}{p'} = \frac{1}{p}$, (e) follows from \eqref{eq:1+nk} and the monotonicity of integration, the linearity of integration gives (f), and the fact that $\F\in\mPb(K;\X)$ yields (g).

    Hence, taking the probability measure $\P \coloneq \nu/\nu(\X)$ (guaranteed to exist as $\nu(\X)>0$) settles the statement. \qed

\section{AUXILIARY RESULTS}

In this section we introduce some auxiliary results used in our proofs. Lemma~\ref{lem:f_unbounded_not_q_int} shows that for any non-negative unbounded function there is always a probability measure for which it is exactly $L_1$-integrable (and not $L_q$-integrable with $q>1$), Lemma~\ref{lem:decomp_F} proves a special type of measure decomposition in $[\mPb^p(K;\X)]^0\setminus \{0\}$, Lemma~\ref{lem:decomp_F_rho} proves a similar decomposition in $[\mPb^p(\rho;\X)]^0\setminus\{0\}$, and Lemma~\ref{lem:inj_barycenter} relates the injectivity of the barycenter with the s.n.t.\ of the underlying semimetric space. Lemma~\ref{lem:lp_isom_abs} proves that $\Abs_p(\P)$ and $L_p(\X,\P)$ are isometrically isomorphic, and  Corollary~\ref{lem:abs_holds_ii-iii} uses this fact to show that $\Abs_p(\P)$ is a group with the addition. Lastly, Lemma~\ref{lem:abs_p_subset_mpb} shows a property of the space $\Abs_{p'}(\P)$ when $\P\in\mP^p(K;\X)$, Lemma~\ref{lem:Pbp-vs-Ap} proves that $\mPb^p(K;\X)$ is a subset of $\A^p(K;\X)$, Lemma~\ref{lem:exp_is_ispd} shows that the exponential kernel is $\mPb(K;\R^d)$-i.s.p.d., and Lemma~\ref{lem:kc_ispd} gives us a way to convert a $\mP^{p}(K;\X)$-characteristic kernel to a $\mPb^{p}(K;\X)$-i.s.p.d.\ one.

\begin{lemmaA}[Unbounded $f\ge0$ admits a $\P$ with $f\in L_1(\X,\P)$ but $f\not\in L_q(\X,\P)$ for $q>1$]\label{lem:f_unbounded_not_q_int}
Let $(\X,\tau_\X)$ be a topological space and $q>1$. Then, for any Borel-measurable unbounded function $f:\X\to\Rnn$, there exists a probability measure $\P\in\M_1^+(\X)$ such that \begin{align}
    \int_\X f(x)\d\P(x) <\infty,~ \text{ while }~ \int_\X f^{q}(x)\d\P(x) = \infty. \label{eq:unbounded:req-ed}
\end{align}
\end{lemmaA}

\begin{proof}[Proof (of Lemma~\ref{lem:f_unbounded_not_q_int})]
    Consider an arbitrary $q>1$.
    As the non-negative $f$ is unbounded, there exists a sequence $\{x_n\}_{n\in\Np}\subseteq \X$ such that $\lim_{n\to \infty}f(x_n)=\infty$ and a subsequence $\{x'_n\}_{n\in\Np}$ such that $f^{q-1}(x'_n)>n^2$ for all $n\in\Np$ [and thus $f(x'_n)>1$ for all $n\in \Np$]. Define 
    \begin{align}
        A &\coloneq \sum_{n=1}^\infty \frac{1}{f^{q-1}(x'_n)}\jover a < \sum_{n=1}^\infty \frac{1}{n^2} \jover b<\infty,\label{eq:A}
    \end{align}
    where (a) follows from the definition of $\{x'_n\}_{n\in\Np}$, and (b) holds by the well-known Basel problem. Then, taking the non-negative measure $\nu = \sum_{n=1}^\infty \frac{1}{A f^q\left(x'_n\right)}\delta_{x'_n}$ [notice that $A>0$ and $f^q(x_n')>0$ by $f(x_n')>1$ for all $n\in \Np$, hence is their product $Af^q\left(x'_n\right)$], we have that \begin{align}
        \int_\X f(x) \d\nu(x) = \frac{1}{A} \sum_{n=1}^\infty \frac{f(x'_n)}{f^q(x'_n)} = \frac{1}{A} \underbrace{\sum_{n=1}^\infty \frac{1}{f^{q-1}(x'_n)}}_{\overset{\eqref{eq:A}}{=}A} =1 < \infty,
    \end{align}
    while \begin{align}
        \int_\X f^q(x) \d\nu(x) = \frac{1}{A} \sum_{n=1}^\infty \frac{f^q(x'_n)}{f^q(x'_n)} = \frac{1}{A} \sum_{n=1}^\infty 1 \stackrel{(a)}{=} \infty
    \end{align}
    by using in (a) that $A>0$. To obtain a probability measure $\P\in\M_1^+(\X)$ with the required properties in \eqref{eq:unbounded:req-ed}, it suffices to take $\P \coloneq \nu/\nu(\X)$, after noting that \begin{align}
        \nu(\X) = \int_\X 1 \d\nu(x) = \underbrace{\frac{1}{A}\sum_{n=1}^\infty \frac{1}{f^q(x'_n)}}_{>0} \jover{a}{<} \frac{1}{A} \underbrace{\sum_{n=1}^\infty \frac{1}{f^{q-1}(x'_n)}}_{\overset{\eqref{eq:A}}{=A}}  =1 \implies \nu(\X) \in (0,1),
    \end{align}
    where (a) follows from the fact that  $f^q(x'_n)>f^{q-1}(x'_n)$ since $f(x'_n) > 1$ for all $n\in \Np$. 
\end{proof}

\begin{lemmaA}[A measure in {$[\mPb^p(K;\X)]^0$} has a decomposition in $\mP^p(K;\X)$]\label{lem:decomp_F}
    Let $(\X,\tau_\X)$ be a topological space, $K:\X^2 \to \R$ a kernel, $p\in [1,\infty)$, and $\F\in[\mPb^p(K;\X)]^0\setminus\{0\}$. Then, there exist two distinct measures $\P,\Q\in\mP^p(K;\X)$ and $M\in (0,\infty)$ such that $\F = M(\P - \Q)$.
\end{lemmaA}
\begin{proof}[Proof (of Lemma~\ref{lem:decomp_F})]
Given $\F\in[\mPb^p(K;\X)]^0$,  let its Hahn-Jordan decomposition be $\F=\F^+-\F^-$, where $\F^+$ and $\F^-$ are two non-negative measures satisfying $\F^+(\X) = \F^-(\X)\eqcolon M\in(0,\infty)$ [the equality $\F^+(\X) = \F^-(\X)$ is guaranteed by $\F(\X) = 0$, $M\in(0,\infty)$ comes from $\F\in \M_{b}(\X)\setminus\{0\}$ \citep[Theorem~10.7]{wheeden15measure}].  Further, $\F^+, \F^-\in \mPb^p(K;\X)$ as $\int_\X \norm{K(\cdot,x)}{\H_K}^p\d\F^+(x)\le \int_\X\norm{K(\cdot,x)}{\H_K}^p\d|\F| < \infty$; the claim for $\F^-$ follows similarly.

Then, for the two measures defined as $\P\coloneq\F^+/M$ and $\Q\coloneq \F^-/M$, one has that
\begin{enumerate}[label=(\roman*)]
    \item $\P(\X)= \F^+(\X)/M = M/M= 1$ and $\Q(\X) = \F^-(\X) / M = M/M = 1$.
    \item $\P(A) = \F^+(A)/M\ge0$ and $\Q(A) = \F^-(A)/M\ge0$, for any $A\in\mathcal B_\X$.
    \item By the definition of $\P$ and $\Q$, $\F^+(A)=M\P(A) = \int_A M \d \P(x)$ and $\F^-(A) = M\Q(A) =\int_A M \d \Q(x)$ for any $A\in\mathcal B_\X$; hence
    $\frac{\d \F^+}{\d \P} = \frac{\d \F^-}{\d \Q} = M$. As such,
    \begin{align}
    \infty \jover{a}{>} & \int_\X \norm{K(\cdot,x)}{\H_K}^p\d\F^+(x) \jover{b}{=} M\int_\X \norm{K(\cdot,x)}{\H_K}^p\d\P(x) \text{ and }\\
    \infty \jover{a}{>} & \int_\X \norm{K(\cdot,x)}{\H_K}^p\d\F^-(x) \jover {b}{=} M\int_\X \norm{K(\cdot,x)}{\H_K}^p\d\Q(x),
    \end{align} 
    where (a) holds by $\F^+, \F^- \in \mPb^p(K;\X)$, and (b) is by the change of measure property of Radon-Nikodym derivatives and the homogeneity of the integral. 
\end{enumerate} 
Therefore, $\P,\Q\in\mP^p(K;\X)$ and are such that $\F = \F^+-\F^- = M(\P-\Q)$. Note that by $\F \neq 0$, $\P\neq \Q$.
\end{proof}

\begin{lemmaA}[A measure in {$[\mPb^p(\rho;\X)]^0$} has a decomposition in $\mP^p(\rho;\X)$]\label{lem:decomp_F_rho}
    Let $(\X,\tau_\X)$ a semimetric space, $p\in [1,\infty)$, and $\F\in[\mPb^p(\rho;\X)]^0\setminus\{0\}$. Then, there exist two distinct measures $\P,\Q\in\mP^p(\rho;\X)$ and $M\in (0,\infty)$ such that $\F = M(\P - \Q)$.
\end{lemmaA}

\begin{proof}[Proof (of Lemma~\ref{lem:decomp_F_rho})]
We will proceed similarly as in Lemma~\ref{lem:decomp_F} while guaranteeing that the decomposition remains in $\mP^p(\rho;\X)$.
Given $\F\in[\mPb^p(\rho;\X)]^0$, let its Hahn-Jordan decomposition be $\F=\F^+-\F^-$, where $\F^+$ and $\F^-$ are two non-negative measures satisfying $\F^+(\X) = \F^-(\X)\eqcolon M\in(0,\infty)$ [the equality $\F^+(\X) = \F^-(\X)$ is guaranteed by $\F(\X) = 0$, $M\in(0,\infty)$ comes from $\F\in \M_{b}(\X)\setminus\{0\}$ \citep[Theorem~10.7]{wheeden15measure}].  Further, $\F^+, \F^-\in \mPb^p(\rho;\X)$ as $\int_\X \rho^{p/2}(x,x_0)\d\F^+(x)\le \int_\X\rho^{p/2}(x,x_0)\d|\F| < \infty$, for some $x_0\in \X$; the claim for $\F^-$ follows similarly with some $x_0'\in \X$.

Then, for the two measures defined as $\P\coloneq\F^+/M$ and $\Q\coloneq \F^-/M$, one has that
\begin{enumerate}[label=(\roman*)]
    \item $\P(\X)= \F^+(\X)/M = M/M= 1$ and $\Q(\X) = \F^-(\X) / M = M/M = 1$.
    \item $\P(A) = \F^+(A)/M\ge0$ and $\Q(A) = \F^-(A)/M\ge0$, for any $A\in\mathcal B_\X$.
    \item By the definition of $\P$ and $\Q$, $\F^+(A)=M\P(A) = \int_A M \d \P(x)$ and $\F^-(A) = M\Q(A) =\int_A M \d \Q(x)$ for any $A\in\mathcal B_\X$; hence
    $\frac{\d \F^+}{\d \P} = \frac{\d \F^-}{\d \Q} = M$. As such, for some $x_0,x'_0\in\X$
    \begin{align}
    \infty \jover{a}{>} & \int_\X \rho^{p/2}(x,x_0)\d\F^+(x) \jover{b}{=} M\int_\X \rho^{p/2}(x,x_0)\d\P(x) \text{ and }\\
    \infty \jover{a}{>} & \int_\X \rho^{p/2}(x,x'_0)\d\F^-(x) \jover {b}{=} M\int_\X \rho^{p/2}(x,x'_0)\d\Q(x),
    \end{align} 
    where (a) holds by $\F^+, \F^- \in \mPb^p(\rho;\X)$, and (b) is by the change of measure property of Radon-Nikodym derivatives and the homogeneity of the integral. 
\end{enumerate} 
Therefore, $\P,\Q\in\mP^p(\rho;\X)$ and are such that $\F = \F^+-\F^- = M(\P-\Q)$. Note that by $\F \neq 0$, $\P\neq \Q$.
\end{proof}

The following result is a remark by \citet[Proposition 3.1]{lyons13distance}; we include the proof for completeness.
\begin{lemmaA}[Injectivity of the barycenter implies s.n.t.]
\label{lem:inj_barycenter} For $p\in[2,\infty)$, a semimetric space $(\X,\rho)$ of n.t.\ (as witnessed by the isometric embedding $\phi:(\X,\rho^{1/2})\to\H$, where $\H$ is a Hilbert space) has $[\mPb^p(\rho;\X)]^0$-s.n.t.\  iff $\beta_\phi$ is injective on $\mP^p(\rho;\X)$. 
\end{lemmaA}

\begin{proof}[Proof (of Lemma~\ref{lem:inj_barycenter})]
Indeed, note that for any $\F\in[\mPb^p(\rho;\X)]^0\setminus \{0\}$,
\begin{align}
\int_\X\int_\X\rho(x,y)\d\F(x)\d\F(y)  &\jover a= \int_\X\int_\X \norm{\phi(x)-\phi(y)}{\H}^2\d\F(x)\d\F(y) \\&\jover b= \int_\X\int_\X\norm{\phi(x)}{\H}^2 + \norm{\phi(y)}{\H}^2 - 2\fip{\phi(x),\phi(y)}{\H}\d\F(x)\d\F(y)
\\ &\jover c= -2\int_\X\int_\X\fip{\phi(x),\phi(y)}{\H}\d\F(x)\d\F(y)
\jover d= -2\fip{\beta_\phi(\F),\beta_\phi(\F)}{\H} \jover{e}{=} -2\norm{\beta_\phi(\F)}{\H}^2\\
&\jover{f}{=} -2\norm{\beta_\phi\big(M(\P-\Q)\big)}{\H} \jover{g}{=} -2M\norm{\beta_\phi(\P)-\beta_\phi(\Q)}{\H},\label{eq:rho_snt_iff_inj}
\end{align}
where (a) comes from the definition of isometric embedding, (b) follows from the fact that in a Hilbert space the norm is induced by its inner product and the linearity of the inner product, (c) is by the linearity of integration and the fact that $\F(\X)=0$ (as $\F \in [\mPb^p(\rho;\X)]^0$), and in (d) the inner product was swapped with the integral [\citealt[(A.32)]{steinwart08support}]. (e) follows from the fact that in a Hilbert space the norm is induced by its inner product. Lastly, %
Lemma~\ref{lem:decomp_F_rho} for $\F\ne 0$ gives (f) with distinct $\P,\Q\in\mP^p(\rho;\X)$, while the linearity of Bochner's integral (both in the integrand and the measure) and the absolute homogeneity of the norm combined with $M>0$ yield (g). 
Hence, as $M\ne 0$, $(\X,\rho)$ has $[\mPb^p(\rho;\X)]^0$-s.n.t.\ iff $\beta_\phi$ is injective on $\mP^p(\rho;\X)$. 
\end{proof}

\begin{lemmaA}[$\Abs_{p}(\P)$ and $L_p(\X,\P)$ are isometrically isomorphic]\label{lem:lp_isom_abs}
    Let  $(\X,\tau_\X)$ be a topological space, %
    $K:\X^2\to\R$ a kernel, $\P\in\mP(K;\X)$, and $p\in(1,\infty]$. For $\F \in \Abs_{p}(\P)$, define $\norm{\F}{\P,p}=\norm{\frac{\d\F}{\d\P}}{L_p\coloneq L_p(\X,\P)}$.  Then, $(\Abs_{p}(\P),\norm{\cdot}{\P,p})$ is a normed space. Moreover, $\imath_\P : L_p(\X,\P) \to \Abs_{p}(\P)$ defined by $[f]_\sim \mapsto  \F_{\P,f}$ with $\frac{\d \F_{\P,f}}{\d \P}= f$ is an isometric isomorphism.
    \end{lemmaA}

\begin{proof}[Proof (of Lemma~\ref{lem:lp_isom_abs})] Let $a\in\R$ and $\F,\F_1,\F_2\in \Abs_{p}(\P)$. We first note that $\Abs_{p}(\P)$ is a vector space over $\R$. Indeed, $\M_b(\X)$ is so, hence it suffices to show that $\Abs_{p}(\P)$ is a subspace of $\M_b(\X)$, in other words that (i) $\Abs_{p}(\P)\neq \emptyset$ [for instance, $0 \in \Abs_{p}(\P)$], (ii) $a\F \in \Abs_{p}(\P)$, and (iii) $\F_1+\F_2 \in \Abs_{p}(\P)$; these properties follow from the definition of $\Abs_{p}(\P)$.

Next, we show that $\norm{\cdot}{\P,p}$ is a norm.
\begin{enumerate}[label=(\roman*)]
    \item $\norm{a\F}{\P,p} = \norm{\dfrac{\d (a \F)}{\d\P}}{L_p}  = |a| \norm{\dfrac{\d \F}{\d\P}}{L_p}=|a|\norm{\F}{\P,p}$, by using the linearity of the Radon-Nikodym derivative and the absolute homogeneity of the $L_p$-norm.
    \item $\norm{\F_1+\F_2}{\P,p} = \norm{\dfrac{\d(\F_1+\F_2)}{\d\P}}{L_p} = \norm{\dfrac{\d\F_1}{\d\P}+\dfrac{\d \F_2}{\d\P}}{L_p} \le \norm{\dfrac{\d\F_1}{\d\P}}{L_p}+\norm{\dfrac{\d \F_2}{\d\P}}{L_p} = \norm{\F_1}{\P,p} +\norm{\F_2}{\P,p}$, by using the linearity of the Radon-Nikodym derivative, three times the definition of $\norm{\cdot}{\P,p}$, and the triangle inequality for the  $\norm{\cdot}{L_p}$ norm.
    \item $\norm{\F}{\P,p}\ge 0$ for any $\F \in \Abs_{p}(\P)$ as the $L_p$ norm is so. $\norm{\F}{\P,p} = 0$ iff $\F$ = 0 follows from the definition of $\norm{\cdot}{\P,p}$ and the fact that $\norm{\cdot}{L_p}$ is a norm. Indeed, 
    $\F=0$ implies that $\frac{\d \F}{\d \P}= 0$, hence $\norm{\F}{\P,p}=\norm{\frac{\d\F}{\d\P}}{L_p}=0$. For $\F\ne 0$, $\dfrac{\d\F}{\d\P} \neq 0$ with positive $\P$-probability; therefore $\norm{\F}{\P,p} = \norm{\dfrac{\d\F}{\d\P}}{L_p} > 0$.
\end{enumerate}

It remains to prove that $\imath_\P$ is an isometric isomorphism, for which it suffices to establish that $\imath_\P$ is a linear isometric surjection (see \citealt[page~46--47]{robinson20functionalanalysis}); this can be seen as follows. 
\begin{enumerate}[label=(\roman*)]
        \item Let $a,b \in \R$ and $[f]_\sim,[g]_\sim \in L_p$. Linearity holds as $\imath_\P(a[f]_\sim + b[g]_\sim) = \imath_\P([af+bg]_\sim) = \F_{\P, af+bg} = a\F_{\P,f} + b\F_{\P,g} = a\imath_\P([f]_\sim)+b\imath_\P([g]_\sim])$, where we used that
        \begin{align}
        \frac{\d \F_{\P,af+bg}}{\d\P} &= af+bg = a\frac{\d \F_{\P,f}}{\d\P}+b\frac{\d \F_{\P,g}}{\d\P}.
        \end{align}
        \item $\imath_\P$ is surjective by definition.
        \item Further, $\imath_\P$ is an isometry as $
          \norm{\imath_\P([f]_\sim)}{\P,p} = \norm{f}{L_p}$ holds by the definition of $\imath_\P$ and that of the $\norm{\cdot}{\P,p}$-norm.
    \end{enumerate}
This completes the proof.
\end{proof}

\begin{corollaryA}[$\Abs_{p'}(\P)$ is a group]\label{lem:abs_holds_ii-iii}
Let $p'\in (1,\infty]$, $(\X,\tau_\X)$ a topological space, 
$K:\X^2\to\R$ a kernel, and $\P\in\mP^p(K;\X)$. Then, $(\Abs_{p'}(\P),+)$ is a group. 
\end{corollaryA}

\begin{proof}[Proof (of Corollary~\ref{lem:abs_holds_ii-iii})]
By Lemma~\ref{lem:lp_isom_abs}, $\Abs_{p'}(\P)$ is a vector space; hence $(\Abs_{p'}(\P),+)$ is a group.
\end{proof}

\begin{lemmaA}[A property of the space $\Abs_{p'}(\P)$ for measures in $\mP^p(K;\X)$]\label{lem:abs_p_subset_mpb}
Let $(\X,\tau_\X)$ be a topological space,
$K:\X^2\to\R$ a kernel, and $p\in [1,\infty)$. Then, for any $\P\in\mP^p(K;\X)$, one has  that $\Abs_{p'}(\P) \subseteq \mPb(K;\X)$, implying that $\A^p(K;\X)\subseteq \mPb(K;\X)$.
\end{lemmaA}

\begin{proof}[Proof (of Lemma~\ref{lem:abs_p_subset_mpb})]
    Let $\F\in\Abs_{p'}(\P)$, in other words $\F\in\M_b(\X)$, $\F\ll\P$ and $\frac{\d \F}{\d \P} \eqcolon f\in L_{p'}(\X,\P)$. This implies that
    \begin{align}
    \int_\X \norm{K(\cdot,x)}{\H_K} \d |\F|(x) 
    & \jover a= \int_\X \norm{K(\cdot,x)}{\H_K} |f(x)|\d \P(x)\\
    & \jover b\leq \left(\int_\X \norm{K(\cdot,x)}{\H_K}^p\!\!\d \P(x)\right)^{1/p}\left(\int_\X  |f(x)|^{p'}\d \P(x)\right)^{1/p'} \jover c< \infty, \label{eq:F-in-Pb}
    \end{align}
    where (a) holds by the fact that $\frac{\d |\F|}{\d \P} = \left|\frac{\d \F}{\d \P}\right|$ as $\F$ is a finite signed measure \citep[Theorem 6.13]{rudin87real},
    (b) follows from the Hölder's inequality, and (c) is implied by $\P\in\mP^p(K;\X)$ and $f\in L_{p'}(\X,\P)$. \eqref{eq:F-in-Pb} means that $|\F|\in\mPb(K;\X)$, which implies that $\F\in\mPb(K;\X)$; thus $\Abs_{p'}(\P)\subseteq \mPb(K;\X)$. The implication in the statement holds by the definition of $\A^p(K;\X)$.
\end{proof}

\begin{lemmaA}[$\mPb^p(K;\X) \subseteq \A^p(K;\X)$] \label{lem:Pbp-vs-Ap}
Let $p\in [1,\infty)$, $(\X,\tau_\X)$ a topological space,  
$K:\X^2\to\R$ a kernel. 
Then, $\mPb^p(K;\X) \subseteq \A^p(K;\X)$. 
\end{lemmaA}

\begin{proof}[Proof (of Lemma~\ref{lem:Pbp-vs-Ap})]
Let $\F \in \mPb^p(K;\X)$. If $\F = 0$, then any $\P\in \mP^p(K;\X)$ is such that $0\ll \P$ with $\frac{\d 0}{\d \P} = 0\in L_{p'}(\X,\P)$, implying that $0\in \A^p(K;\X)$; hence, in the sequel, we assume that $\F\neq0$. Then, with $\P_\F \coloneq\frac{|\F|}{|\F|(\X)}\in \mP^p(K;\X)$ [the inclusion holds by the definition of $\mPb^p(K;\X)$],\footnote{Note that $|\F|(\X)>0$ as $\F\neq0$.} for $\P_\F$-a.e.\ $x\in \X$,  we have that 
\begin{align}
 \left|\frac{\d \F}{\d\P_\F}(x)\right| \jover a= \left|\frac{\d \F}{\d|\F|}(x)\right|\left|\frac{\d|\F|}{\d \P_\F}(x)\right| \jover b= |\F|(\X) <\infty,
\end{align} 
 where (a) comes from the chain rule of the Radon-Nikodym derivative (using that $\F \ll |\F| \ll \P_\F$; \citealt[Proposition 3.9.b]{folland99real}), and (b) is implied by the facts that $\left|\frac{\d \F}{\d |\F|}(x)\right|=1$ for $\F$-a.e.\ $x\in\X$ \citep[Exercise~13.7]{bass11real} and $\frac{\d|\F|}{\d \P_\F}=|\F|(\X)$; this means that $\frac{\d \F}{\d \P_\F}$ is $\P_\F$-essentially bounded: $\frac{\d \F}{\d \P_\F} \in L_\infty(\X,\P_\F)\subseteq L_{p'}(\X,\P_\F)$, proving that $\F\in\Abs_{p'}(\P_\F)\subseteq \A^p(K;\X)$. 
\end{proof}

\begin{lemmaA}[The exponential kernel is $\mPb(K;\R^d)$-i.s.p.d.]\label{lem:exp_is_ispd}
    The exponential kernel [$K(\b x,\b y) = e^{a\fip{\b x,\b y}{}}$ with $\b x, \b y \in \R^d$ and $a>0$] is $\mPb(K;\R^d)$-i.s.p.d.
\end{lemmaA}
\begin{proof}[Proof (of Lemma~\ref{lem:exp_is_ispd})]
    Consider an arbitrary $\F \in \mPb(K;\R^d)\setminus \{0\}$, and observe that
    \begin{align}
        \hspace*{-0.3cm}\infty \jover{Def.}{>}  \int_{\R^d} \norm{K(\b x, \cdot)}{\H_K} \d |\F|(\b x) \jover a=  \int_{\R^d} \sqrt{K(\b x,\b x)} \d |\F|(\b x) \jover b =  \int_{\R^d} e^{\frac{a\fip{\b x,\b x}{}}{2}} \d |\F|(\b x) \jover c =  \int_{\R^d} e^{\frac{a\norm{\b x}{2}^2}{2}} \d |\F|(\b x), \label{eq:exp_finite_int}
    \end{align}
    where (a) follows from the fact that in a Hilbert space the norm is induced by its inner product and the reproducing property, (b) holds by the definition of the exponential kernel and the fact that $\left(e^b\right)^c = e^{bc}$ for $b,c\in \R$, and (c) is by the fact that $\fip{\b x,\b x}{} = \norm{\b x}{2}^2$. 
    
    We introduce three auxiliary signed measures which will be useful in the forthcoming. Indeed, with $f(\b x)\coloneq e^{\frac{a\norm{\b x}{2}^2}{2}}$ ($\b x\in \R^d$), define the signed measures
    \begin{align}
    \F_{1}(A) &\coloneq \int_{A} f(\b x) \d \F^+(\b x), &
    \F_2(A) &\coloneq \int_{A} f(\b x) \d\F^-(\b x), \, (A\in\mB_{\R^d})& \F' &\coloneq \F_1 - \F_2; \label{eq:def:F_1-F_2-F'}
    \end{align}
    these measures are guaranteed to exists as $f\in L_1(\R^d,|\F|)$ by \eqref{eq:exp_finite_int}, and \eqref{eq:def:F_1-F_2-F'} implies that 
    \begin{align}
    \frac{\d \F_1}{\d \F^+}&=f,&
    \frac{\d \F_2}{\d \F^-}&=f. \label{eq:dF1-or-F2-per-dFpm}
    \end{align}
    
    Observe that for any $g\in L_\infty(\R^d,|\F|)$,\begin{align}
        \norm{f g}{L_1(\R^d,|\F|)} \jover a \leq \norm{f}{L_1(\R^d,|\F|)}\norm{g}{L_\infty(\R^d,|\F|)} \jover b < \infty, \label{eq:well_def_change_measure}
    \end{align}
    where (a) comes from Hölder's inequality and (b) holds since $f\in L_1(\R^d,|\F|)$ [by \eqref{eq:exp_finite_int}] and $g\in L_\infty(\R^d,|\F|)$. Then, we have the following change of measure formula:
    \begin{align}
        \infty \overset{\eqref{eq:well_def_change_measure}}{>}\int_{\R^d} f(\b x) g(\b x) \d \F(\b x) 
        &\jover a = \int_{\R^d} f(\b x) g(\b x) \d \F^+(\b x) - \int_{\R^d}  f(\b x) g(\b x) \d \F^-(\b x) 
        \\
        &\jover b =  \int_{\R^d} g(\b x) \d \F_1(\b x) - \int_{\R^d}  g(\b x) \d \F_2(\b x)
        \\&\jover c = \int_{\R^d}g(\b x) \d[\F_1-\F_2](\b x) \jover {d}{=} \int_{\R^d}g(\b x) \d\F'(\b x), \label{eq:exp_change_meas}
    \end{align}
    where (a) comes from the definition of integration w.r.t\ signed measures, (b) is by the change of measure property of Radon-Nikodym derivatives \citep[Proposition 3.9]{folland99real} [using \eqref{eq:dF1-or-F2-per-dFpm}; valid as $g\in L_\infty(\R^d,|\F|)\subseteq L_1(\R^d,|\F|)$], (c) holds by the linearity of integration w.r.t.\ the measure, and (d) follows from the definition of $\F'$.
    
    Next we show that the double integral of $K$ w.r.t.\ $\F$ is positive; as $\F\in\mPb(K;\R^d)\setminus\{0\}$ is arbitrary, this implies the $\mPb(K;\R^d)$-i.s.p.d.\ property we seek to prove. Indeed, 
    \begin{align}
        \MoveEqLeft\int_{\R^d}\int_{\R^d} e^{a\fip{\b x,\b y}{}} \d \F(\b x)\d\F(\b y) \jover{a}{=}  \int_{\R^d}\int_{\R^d} e^{\frac{a}{2}\big(\norm{\b x}{2}^2+\norm{\b y}{2}^2 -\norm{\b x-\b y}{2}^2\big)} \d \F(\b x)\d\F(\b y)  \\
        &\jover{b}{=}\int_{\R^d}\int_{\R^d} e^{\frac{a}{2}\norm{\b x}{2}^2}e^{\frac{a}{2}\norm{\b y}{2}^2}e^{-\frac{a}{2}\norm{\b x-\b y}{2}^2} \d \F(\b x)\d\F(\b y) 
        \jover {c}{=}\int_{\R^d}e^{\frac{a}{2}\norm{\b y}{2}^2}\Bigg(\int_{\R^d} e^{\frac{a}{2}\norm{\b x}{2}^2}e^{-\frac{a}{2}\norm{\b x-\b y}{2}^2}\d \F(\b x)\Bigg)\d\F(\b y) 
        \\&\jover {d}{=}\int_{\R^d}e^{\frac{a}{2}\norm{\b y}{2}^2}\Bigg(\int_{\R^d}e^{-\frac{a}{2}\norm{\b x-\b y}{2}^2}\d \F'(\b x)\Bigg)\d\F(\b y) 
        \jover{c}{=}\int_{\R^d}\int_{\R^d}e^{\frac{a}{2}\norm{\b y}{2}^2}e^{-\frac{a}{2}\norm{\b x-\b y}{2}^2}\d \F'(\b x)\d\F(\b y) 
        \\&\jover{e}{=}\int_{\R^d}\int_{\R^d}e^{\frac{a}{2}\norm{\b y}{2}^2}e^{-\frac{a}{2}\norm{\b x-\b y}{2}^2}\d \F(\b y)\d\F'(\b x) 
        \jover{f}{=}\int_{\R^d}\int_{\R^d} e^{-\frac{a}{2}\norm{\b x-\b y}{2}^2}\d \F'(\b y)\d\F'(\b x) \jover g > 0, \label{eq:exp_is_ispd}
    \end{align}
    where (a) holds as $\norm{\b x-\b y}{2}^2 = \norm{\b x}{2}^2  + \norm{\b y}{2}^2- 2\fip{\b x,\b y}{}$, and hence  $\fip{\b x,\b y}{} =  \frac{1}{2}\big(\norm{\b x}{2}^2+\norm{\b y}{2}^2 -\norm{\b x-\b y}{2}^2\big)$, while the properties of exponential function ($e^{a+b}=e^{a} e^{b}$ with $a,b\in \R$) give (b), the homogeneity of the integral yields (c), (d) follows from applying \eqref{eq:exp_change_meas} inside the parentheses [valid as $f\coloneq e^{\frac{a}{2}\norm{\cdot}{2}^2} \in L_1(\R^d,|\F|)$ by \eqref{eq:exp_finite_int} and $g_{\b y} \coloneq e^{-\frac{a}{2}\norm{\cdot -\b y}{2}^2}\in L_\infty(\R^d,|\F|)$ as it is bounded for any fixed $\b y \in \R^d$], (e) is by applying Tonelli's theorem [recalled in Theorem~\ref{th:Fubini}; valid as $F(\b x, \b y) \coloneq f(\b y) e^{-\frac{a}{2}\norm{\b x - \b y}{2}^2}$ is positive for all $(\b x,\b y) \in \R^d\times \R^d$], and (f) holds by applying \eqref{eq:exp_change_meas} again in the inner integral [with the same choice of $f$ and $g_{\b y}$ as in (d)]. Lastly, we note that $k(\b x,\b y) \coloneq e^{-\frac{a}{2}\norm{\b x - \b y}{2}^2}$ ($\b x,\b y\in \R^d$) is the well-known Gaussian kernel with $\gamma = a/2$. The Gaussian kernel is $\M_1^+(\R^d)$-characteristic \citep[Theorem 9 and Table 2]{sriperumbudur10hilbert}, and hence $\M_b(\R^d)$-i.s.p.d.\ [\citealt[Figure 1(4)]{sriperumbudur11universality}], proving (g) (valid as $\F'\neq 0$ since $f>0$ and $\F\neq0$ by assumption).
\end{proof}

The following lemma extends a known result by \citet[Theorem~8(i)-(iii)]{simon-gabriel18kernel} to the unbounded case.
\begin{lemmaA}[Shifted $\mP^p(K;\X)$-characteristic kernels are $\mPb^p(K;\X)$-i.s.p.d.]\label{lem:kc_ispd} 
Let $(\X,\tau_\X)$ be a topological space, $p\in[1,\infty)$, $c>0$, and $K:\X\times \X\to \R$ a $\mP^p(K;\X)$-characteristic kernel. Then, for the kernel $K_c \coloneq K + c$, $\mPb(K;\X) = \mPb(K_c;\X)$ and $K_c$ is $\mPb^p(K_c;\X)$-i.s.p.d.\footnote{Note that $K_c$ is known to be a kernel.}
\end{lemmaA}
\begin{proof}[Proof (of Lemma~\ref{lem:kc_ispd})]
We first show that $\mPb^p(K;\X) = \mPb^p(K_c;\X)$. 

\begin{itemize}
    \item[($\subseteq$)]   Consider an arbitrary $\F\in\mPb^p(K;\X)$. In this case, we have 
    \begin{align}
        \MoveEqLeft\Bigg(\int_\X K_c^{p/2}(x,x) \d|\F|(x)\Bigg)^{1/p} 
        = \Bigg(\int_\X \big[K(x,x) + c\big]^{p/2}\d|\F| (x)\Bigg)^{1/p}\jover a \leq \Bigg(\int_\X \Big[\sqrt{K(x,x)} + \sqrt{c}\Big]^{p}\d|\F| (x)\Bigg)^{1/p}
        \\
        &\jover b \leq \Bigg( \underbrace{\int_\X K^{p/2}(x,x)\d|\F| (x)}_{\eqcolon A}\Bigg)^{1/p} + \Bigg(\int_\X c^{p/2}\d|\F| (x)\Bigg)^{1/p}
        \jover c = A^{1/p} + \underbrace{c^{1/2}|\F|^{1/p}(\X)}_{\eqcolon B} \jover d< \infty,
    \end{align}
    where (a) follows from $\sqrt{a+b}\leq \sqrt{a} + \sqrt{b}$ holding for $a,b\ge 0$, (b) comes from the well-known Minkowski inequality, and (c) holds by the homogeneity of the integral. $B<\infty$ as $\F$ is a finite measure and $c>0$, and $A<\infty$ as $\F\in\mPb^p(K;\X)$, giving (d). Therefore, $\F\in\mPb^p(K_c;\X)$.
    
    \item[($\supseteq$)] Consider an arbitrary $\F\in\mPb^p(K_c;\X)$. Then, 
    \begin{align}
        \infty \jover a > \int_\X K_c^{p/2}(x,x) \d|\F|(x) = \int_\X \big[K(x,x) + c\big]^{p/2} \d|\F|(x) \jover{b}{ \ge} \int_{\X} K^{p/2}(x,x)\d|\F|(x),
    \end{align}
    where (a) comes from $\F\in\mPb^p(K_c;\X)$, and (b) is implied by $K(x,x) + c > K(x,x)$ (since $c>0$) and the monotonicity of integration. This means that $\F\in\mPb^p(K;\X)$.
\end{itemize}
Having established the equality $\mPb^p(K;\X) = \mPb^p(K_c;\X)$, we now proceed with proving our  statement. Indeed, for any $\F \in \mPb^p(K_c;\X)\setminus\{0\}$, 
\begin{align}
     \int_\X \int_\X K_c(x,y) \d \F(x)\d \F(y) \jover{a}{=} \int_\X \int_\X K(x,y) + c~ \d \F(x)\d \F(y) \jover{b}{=} \underbrace{\norm{\mu_K(\F)}{\H_K}^2}_{\eqcolon C} + \underbrace{c[\F(\X)]^2}_{\eqcolon D}, \label{eq:kc_ispd}
\end{align}
where (a) holds by the definition of $K_c$, (b) follows from the linearity of the integral and \eqref{eq:def:ispd}. Note that if $\F(\X)\neq 0$, then $D>0$ by $c>0$, proving the statement; therefore, assume that $\F(\X)= 0$. In this case, $D=0$ and $\F\in[\mPb^p(K;\X)]^0\setminus\{0\}$ [as $\mPb^p(K_c;\X) = \mPb^p(K;\X)$].  Note that since $K$ is $\mP^p(K;\X)$-characteristic, it is also $[\mPb^p(K;\X)]^0$-characteristic (by Theorem~\ref{th:char_p1_iff_pb0}); hence, it is $[\mPb^p(K;\X)]^0$-i.s.p.d.\ [by Theorem~\ref{th:ispd0_iff_char}, using that $([\mPb^p(K;\X)]^0,+)$ is a group]. Therefore, $C+D = C = \norm{\mu_K(\F)}{\H_K}^{2}>0$ in \eqref{eq:kc_ispd} for any $\F\in\mPb^p(K_c;\X)\setminus\{0\}$, proving the statement. 
\end{proof}

\section{EXTERNAL RESULTS}

In this section we include the main external results used through our proofs. Proposition~\ref{prop:equiv_int_k_rho} relates the sets $\mPb^p(K;\X)$ and $\mPb^p(\rho_{K};\X)$, Theorem~\ref{th:Fubini} contains the well-known Fubini-Tonelli theorem, Theorem~\ref{th:ingo_4.26} shows important properties of the $S_K$ operator, and, lastly, Lemma~\ref{lem:hoelder} contains a generalized version of Hölder's inequality. 

\begin{propositionA}[Integrability of kernels and semimetrics; \citealt{sejdinovic13equivalence}, Proposition 20]\label{prop:equiv_int_k_rho}
    Let $K:\X^2\to \R$ be a kernel that generates the semimetric $\rho_K$, and let $p\in\N$. Then $\mPb^p(K;\X)=\mPb^p(\rho_K;\X)$.
\end{propositionA}

\begin{theoremA}[Fubini-Tonelli; \citealt{folland99real}, Theorem 2.37.b]\label{th:Fubini}
    Suppose that $(\mathcal{X},\mathcal{M},\nu_1)$ and $(\mathcal{Y},\mathcal{N},\nu_2)$ are $\sigma$-finite (non-negative) measure spaces.  If $ F:  \mathcal{X}\times \mathcal{Y} \to \Rnn$, then
    \begin{equation}
        \int_{\mathcal{X}\times \mathcal{Y}} F(x,y)\d(\nu_1\otimes\nu_2)(x, y) = \int_\mathcal{X} \left[\int_\mathcal{Y} F(x,y) \d \nu_2(y) \right]\d \nu_1 (x) = \int_\mathcal{Y} \left[\int_\mathcal{X} F(x,y) \d \nu_1(x) \right]\d \nu_2 (y).
    \end{equation}
     
\end{theoremA}

\begin{theoremA}[Properties of $S_K$; \citealt{steinwart08support}, Theorem~4.26]\label{th:ingo_4.26} 
    Let $(\X,\tau_\X)$ be a topological space, $p\in[1,\infty)$, $\P\in\mP^p(K;\X)$, and $\H_K$ a separable RKHS over $\X$ with kernel $K:\X^2\to\R$. Then, $\H_K\subseteq L_p(\X,\P)$, and $\H_K$ is dense in $L_p(\X,\P)$ iff $S_K:L_{p'}(\X,\P)\to\H_K$ is injective.  
\end{theoremA}

\begin{lemmaA}[A generalized Hölder's inequality; \citealt{folland99real}, Exercise 6.31]\label{lem:hoelder}
    Let $(\X,\mB_\X,\nu)$ be a measure space with a  finite non-negative measure $\nu$, $p_m\in[1,\infty]$ ($m\in [M]$; with the convention $\frac{1}{\infty}\coloneq 0$) and $\sum_{m= 1}^M \frac{1}{p_m} \eqcolon \frac{1}{r} \leq 1$. If $f_m\in L_{p_m}(\X,\nu)$ for $m\in[M]$, then $\prod_{m= 1}^M f_m \in L_r(\X,\nu)$ and $\norm{\prod_{m= 1}^M f_m }{L_r}\le \prod_{m= 1}^M \norm{f_m}{L_{p_m}}$.  
\end{lemmaA}

%% file: main_arxiv.bbl
\begin{thebibliography}{51}
\providecommand{\natexlab}[1]{#1}
\providecommand{\url}[1]{\texttt{#1}}
\expandafter\ifx\csname urlstyle\endcsname\relax
  \providecommand{\doi}[1]{doi: #1}\else
  \providecommand{\doi}{doi: \begingroup \urlstyle{rm}\Url}\fi

\bibitem[Anastasiou et~al.(2023)Anastasiou, Barp, Briol, Ebner, Gaunt,
  Ghaderinezhad, Gorham, Gretton, Ley, Liu, Mackey, Oates, Reinert, and
  Swan]{anastasiou23stein}
Andreas Anastasiou, Alessandro Barp, Fran\c{c}ois-Xavier Briol, Bruno Ebner,
  Robert~E. Gaunt, Fatemeh Ghaderinezhad, Jackson Gorham, Arthur Gretton,
  Christophe Ley, Qiang Liu, Lester Mackey, Chris~J. Oates, Gesine Reinert, and
  Yvik Swan.
\newblock {S}tein's method meets computational statistics: a review of some
  recent developments.
\newblock \emph{Statistical Science}, 38\penalty0 (1):\penalty0 120--139, 2023.

\bibitem[Aronszajn(1950)]{aronszajn50theory}
Nachman Aronszajn.
\newblock Theory of reproducing kernels.
\newblock \emph{Transactions of the American Mathematical Society},
  68:\penalty0 337--404, 1950.

\bibitem[Baringhaus and Franz(2004)]{baringhaus04new}
Ludwig Baringhaus and C.~Franz.
\newblock On a new multivariate two-sample test.
\newblock \emph{Journal of Multivariate Analysis}, 88:\penalty0 190--206, 2004.

\bibitem[Barp et~al.(2024)Barp, Simon-Gabriel, Girolami, and
  Mackey]{barp24targeted}
Alessandro Barp, Carl-Johann Simon-Gabriel, Mark Girolami, and Lester Mackey.
\newblock Targeted separation and convergence with kernel discrepancies.
\newblock \emph{Journal of Machine Learning Research}, 25\penalty0
  (378):\penalty0 1--50, 2024.

\bibitem[Bass(2024)]{bass11real}
Richard~F. Bass.
\newblock \emph{Real Analysis for Graduate Students: Measure and Integration
  Theory}.
\newblock 2024.

\bibitem[Berlinet and Thomas-Agnan(2004)]{berlinet04reproducing}
Alain Berlinet and Christine Thomas-Agnan.
\newblock \emph{Reproducing Kernel Hilbert Spaces in Probability and
  Statistics}.
\newblock Kluwer, 2004.

\bibitem[Carmeli et~al.(2010)Carmeli, Vito, Toigo, and
  Umanit{\'a}]{carmeli10vector}
Claudio Carmeli, Ernesto~De Vito, Alessandro Toigo, and Veronica Umanit{\'a}.
\newblock Vector valued reproducing kernel {H}ilbert spaces and universality.
\newblock \emph{Analysis and Applications}, 8:\penalty0 19--61, 2010.

\bibitem[Chen(2021)]{chen21stein}
Louis H.~Y. Chen.
\newblock {S}tein's method of normal approximation: Some recollections and
  reflections.
\newblock \emph{Annals of Statistics}, 49\penalty0 (4):\penalty0 1850--1863,
  2021.

\bibitem[Chwialkowski et~al.(2016)Chwialkowski, Strathmann, and
  Gretton]{chwialkowski16kernel}
Kacper Chwialkowski, Heiko Strathmann, and Arthur Gretton.
\newblock A kernel test of goodness of fit.
\newblock In \emph{International Conference on Machine Learning (ICML)}, pages
  2606--2615, 2016.

\bibitem[Cribeiro-Ramallo et~al.(2026{\natexlab{a}})Cribeiro-Ramallo, Aich,
  Kalinke, Aich, and Szab{\'o}]{cribeiro26minimax}
Jose Cribeiro-Ramallo, Agnideep Aich, Florian Kalinke, Ashit~Baran Aich, and
  Zolt{\'a}n Szab{\'o}.
\newblock The minimax lower bound of kernel {S}tein discrepancy estimation.
\newblock In \emph{International Conference on Artificial Intelligence and
  Statistics (AISTATS)}, pages 901--909, 2026{\natexlab{a}}.

\bibitem[Cribeiro-Ramallo et~al.(2026{\natexlab{b}})Cribeiro-Ramallo, Kalinke,
  and Szab{\'o}]{cribeiro26minimax2}
Jose Cribeiro-Ramallo, Florian Kalinke, and Zolt{\'a}n Szab{\'o}.
\newblock Minimax lower bounds of kernel discrepancy estimation: {MMD}, {HSIC},
  {KSD}.
\newblock Technical report, 2026{\natexlab{b}}.
\newblock (\url{https://arxiv.org/abs/2607.24235}).

\bibitem[Diestel and Uhl(1977)]{diestel77vector}
Joseph Diestel and John~Jerry Uhl.
\newblock \emph{Vector Measures}.
\newblock American Mathematical Society. Providence, 1977.

\bibitem[Dudley(2004)]{dudley04real}
Richard Dudley.
\newblock \emph{Real Analysis and Probability}.
\newblock Cambridge University Press, 2004.

\bibitem[Folland(1999)]{folland99real}
Gerald~B. Folland.
\newblock \emph{Real Analysis -- Modern Techniques and Their Applications}.
\newblock John Wiley \& Sons, 1999.

\bibitem[Fukumizu et~al.(2008)Fukumizu, Gretton, Sun, and
  Sch{\"o}lkopf]{fukumizu08kernel2}
Kenji Fukumizu, Arthur Gretton, Xiaohai Sun, and Bernhard Sch{\"o}lkopf.
\newblock Kernel measures of conditional dependence.
\newblock In \emph{Advances in Neural Information Processing Systems
  (NeurIPS)}, pages 498--496, 2008.

\bibitem[Gogolashvili(2026)]{gogolashvili26minimax}
Davit Gogolashvili.
\newblock Minimax estimation of kernel {S}tein discrepancy: Trace versus
  {H}ilbert-{S}chmidt scales.
\newblock Technical report, 2026.
\newblock (\url{https://arxiv.org/abs/2607.03367}).

\bibitem[Gretton et~al.(2005{\natexlab{a}})Gretton, Bousquet, Smola, and
  Sch{\"o}lkopf]{gretton05measuring}
Arthur Gretton, Olivier Bousquet, Alex Smola, and Bernhard Sch{\"o}lkopf.
\newblock Measuring statistical dependence with {H}ilbert-{S}chmidt norms.
\newblock In \emph{Algorithmic Learning Theory (ALT)}, pages 63--78,
  2005{\natexlab{a}}.

\bibitem[Gretton et~al.(2005{\natexlab{b}})Gretton, Herbrich, Smola, Bousquet,
  and Sch{\"o}lkopf]{gretton05kernel}
Arthur Gretton, Ralf Herbrich, Alexander Smola, Olivier Bousquet, and Bernhard
  Sch{\"o}lkopf.
\newblock Kernel methods for measuring independence.
\newblock \emph{Journal of Machine Learning Research}, 6\penalty0
  (70):\penalty0 2075--2129, 2005{\natexlab{b}}.

\bibitem[Gretton et~al.(2012)Gretton, Borgwardt, Rasch, Sch{\"o}lkopf, and
  Smola]{gretton12kernel}
Arthur Gretton, Karsten Borgwardt, Malte Rasch, Bernhard Sch{\"o}lkopf, and
  Alexander Smola.
\newblock A kernel two-sample test.
\newblock \emph{Journal of Machine Learning Research}, 13\penalty0
  (25):\penalty0 723--773, 2012.

\bibitem[Hagrass et~al.(2026)Hagrass, Sriperumbudur, and
  Balasubramanian]{hagrass26stein}
Omar Hagrass, Bharath Sriperumbudur, and Krishnakumar Balasubramanian.
\newblock Minimax optimal goodness-of-fit testing with kernel {S}tein
  discrepancy.
\newblock \emph{Bernoulli}, 32\penalty0 (1):\penalty0 299--324, 2026.

\bibitem[Kalinke et~al.(2025)Kalinke, Szab{\'{o}}, and
  Sriperumbudur]{kalinke25nystromksd}
Florian Kalinke, Zolt{\'{a}}n Szab{\'{o}}, and Bharath~K. Sriperumbudur.
\newblock Nyström kernel {S}tein discrepancy.
\newblock In \emph{International Conference on Artificial Intelligence and
  Statistics (AISTATS)}, pages 388--396, 2025.

\bibitem[Klebanov(2005)]{klebanov05ndistance}
Lev Klebanov.
\newblock \emph{N-Distances and Their Applications}.
\newblock Charles University, Prague, 2005.

\bibitem[Kozarzewski(2018)]{kozarzewski18existence}
Piotr~A. Kozarzewski.
\newblock On existence of the support of a {B}orel measure.
\newblock \emph{Demonstratio Mathematica}, 51\penalty0 (1):\penalty0 76--84,
  2018.

\bibitem[Liu et~al.(2016)Liu, Lee, and Jordan]{liu16kernelized}
Qiang Liu, Jason Lee, and Michael Jordan.
\newblock A kernelized {Stein} discrepancy for goodness-of-fit tests.
\newblock In \emph{International Conference on Machine Learning (ICML)}, pages
  276--284, 2016.

\bibitem[Lyons(2013)]{lyons13distance}
Russell Lyons.
\newblock Distance covariance in metric spaces.
\newblock \emph{The Annals of Probability}, 41:\penalty0 3284--3305, 2013.

\bibitem[Modeste and Dombry(2024)]{modeste24characterization}
Thibault Modeste and Cl{{\'e}}ment Dombry.
\newblock Characterization of translation invariant {MMD} on {$\R^d$} and
  connections with {W}asserstein distances.
\newblock \emph{Journal of Machine Learning Research}, 25\penalty0
  (237):\penalty0 1--39, 2024.

\bibitem[Paulsen and Raghupathi(2016)]{paulsen16introduction}
Vern~I. Paulsen and Mrinal Raghupathi.
\newblock \emph{An Introduction to the Theory of Reproducing Kernel Hilbert
  Spaces}.
\newblock Cambridge University Press, 2016.

\bibitem[Pfister et~al.(2018)Pfister, B{\"u}hlmann, Sch{\"o}lkopf, and
  Peters]{pfister18kernel}
Niklas Pfister, Peter B{\"u}hlmann, Bernhard Sch{\"o}lkopf, and Jonas Peters.
\newblock Kernel-based tests for joint independence.
\newblock \emph{Journal of the Royal Statistical Society Series B: Statistical
  Methodology}, 80\penalty0 (1):\penalty0 5--31, 2018.

\bibitem[Quadrianto et~al.(2009)Quadrianto, Song, and
  Smola]{quadrianto09kernelized2}
Novi Quadrianto, Le~Song, and Alex Smola.
\newblock Kernelized sorting.
\newblock In \emph{Advances in Neural Information Processing Systems
  (NeurIPS)}, pages 1289--1296, 2009.

\bibitem[Robinson(2020)]{robinson20functionalanalysis}
James~C. Robinson.
\newblock \emph{An Introduction to Functional Analysis}.
\newblock Cambridge University Press, 2020.

\bibitem[Rudin(1987)]{rudin87real}
Walter Rudin.
\newblock \emph{Real and complex analysis}.
\newblock McGraw-Hill Book Co., New York, 3rd edition, 1987.

\bibitem[Saitoh and Sawano(2016)]{saitoh16theory}
Saburou Saitoh and Yoshihiro Sawano.
\newblock \emph{Theory of Reproducing Kernels and Applications}.
\newblock Springer Singapore, 2016.

\bibitem[Sch{\"o}lkopf et~al.(1998)Sch{\"o}lkopf, Smola, and
  M{\"u}ller]{scholkopf98nonlinear}
Bernhard Sch{\"o}lkopf, Alexander Smola, and Klaus-Robert M{\"u}ller.
\newblock Nonlinear component analysis as a kernel eigenvalue problem.
\newblock \emph{Neural Computation}, 10\penalty0 (5):\penalty0 1299--1319,
  1998.

\bibitem[Schrab(2025)]{schrab25discrepancies}
Antonin Schrab.
\newblock A practical introduction to kernel discrepancies: {MMD}, {HSIC} \&
  {KSD}.
\newblock Technical report, 2025.
\newblock (\url{https://arxiv.org/abs/2503.04820}).

\bibitem[Sejdinovic et~al.(2013{\natexlab{a}})Sejdinovic, Gretton, and
  Bergsma]{sejdinovic13kernel2}
Dino Sejdinovic, Arthur Gretton, and Wicher Bergsma.
\newblock A kernel test for three-variable interactions.
\newblock In \emph{Advances in Neural Information Processing Systems
  (NeurIPS)}, pages 1124--1132, 2013{\natexlab{a}}.

\bibitem[Sejdinovic et~al.(2013{\natexlab{b}})Sejdinovic, Sriperumbudur,
  Gretton, and Fukumizu]{sejdinovic13equivalence}
Dino Sejdinovic, Bharath Sriperumbudur, Arthur Gretton, and Kenji Fukumizu.
\newblock Equivalence of distance-based and {RKHS}-based statistics in
  hypothesis testing.
\newblock \emph{Annals of Statistics}, 41:\penalty0 2263--2291,
  2013{\natexlab{b}}.

\bibitem[Simon-Gabriel and Sch{\"o}lkopf(2018)]{simon-gabriel18kernel}
Carl-Johann Simon-Gabriel and Bernhard Sch{\"o}lkopf.
\newblock Kernel distribution embeddings: Universal kernels, characteristic
  kernels and kernel metrics on distributions.
\newblock \emph{Journal of Machine Learning Research}, 19\penalty0
  (44):\penalty0 1--29, 2018.

\bibitem[Simon-Gabriel et~al.(2023)Simon-Gabriel, Barp, Sch{\"o}lkopf, and
  Mackey]{simon-gabriel23metrizing}
Carl-Johann Simon-Gabriel, Alessandro Barp, Bernhard Sch{\"o}lkopf, and Lester
  Mackey.
\newblock Metrizing weak convergence with maximum mean discrepancies.
\newblock \emph{Journal of Machine Learning Research}, 24\penalty0
  (184):\penalty0 1--20, 2023.

\bibitem[Smola et~al.(2007)Smola, Gretton, Song, and
  Sch{\"o}lkopf]{smola07hilbert}
Alexander Smola, Arthur Gretton, Le~Song, and Bernhard Sch{\"o}lkopf.
\newblock A {H}ilbert space embedding for distributions.
\newblock In \emph{Algorithmic Learning Theory (ALT)}, pages 13--31, 2007.

\bibitem[Sriperumbudur et~al.(2010)Sriperumbudur, Gretton, Fukumizu,
  Sch{\"o}lkopf, and Lanckriet]{sriperumbudur10hilbert}
Bharath Sriperumbudur, Arthur Gretton, Kenji Fukumizu, Bernhard Sch{\"o}lkopf,
  and Gert Lanckriet.
\newblock Hilbert space embeddings and metrics on probability measures.
\newblock \emph{Journal of Machine Learning Research}, 11:\penalty0 1517--1561,
  2010.

\bibitem[Sriperumbudur et~al.(2011)Sriperumbudur, Fukumizu, and
  Lanckriet]{sriperumbudur11universality}
Bharath Sriperumbudur, Kenji Fukumizu, and Gert Lanckriet.
\newblock Universality, characteristic kernels and {RKHS} embedding of
  measures.
\newblock \emph{Journal of Machine Learning Research}, 12:\penalty0 2389--2410,
  2011.

\bibitem[Stein(1972)]{stein72bound}
Charles Stein.
\newblock A bound for the error in the normal approximation to the distribution
  of a sum of dependent random variables.
\newblock In \emph{Berkeley Symposium on Mathematical Statistics and
  Probability}, pages 583--602, 1972.

\bibitem[Steinwart(2001)]{steinwart01influence}
Ingo Steinwart.
\newblock On the influence of the kernel on the consistency of support vector
  machines.
\newblock \emph{Journal of Machine Learning Research}, 6\penalty0 (3):\penalty0
  67--93, 2001.

\bibitem[Steinwart and Christmann(2008)]{steinwart08support}
Ingo Steinwart and Andreas Christmann.
\newblock \emph{Support Vector Machines}.
\newblock Springer, 2008.

\bibitem[Steinwart and Ziegel(2021)]{steinwart21strictly}
Ingo Steinwart and Johanna Ziegel.
\newblock Strictly proper kernel scores and characteristic kernels on compact
  spaces.
\newblock \emph{Applied and Computational Harmonic Analysis}, 51:\penalty0
  510--542, 2021.

\bibitem[Szab{\'o} and Sriperumbudur(2018)]{szabo18characteristic2}
Zolt{\'a}n Szab{\'o} and Bharath~K. Sriperumbudur.
\newblock Characteristic and universal tensor product kernels.
\newblock \emph{Journal of Machine Learning Research}, 18\penalty0
  (233):\penalty0 1--29, 2018.

\bibitem[Sz{\'e}kely and Rizzo(2004)]{szekely04testing}
G{\'a}bor Sz{\'e}kely and Maria Rizzo.
\newblock Testing for equal distributions in high dimension.
\newblock \emph{InterStat}, 5:\penalty0 1249--1272, 2004.

\bibitem[Sz{\'e}kely and Rizzo(2005)]{szekely05new}
G{\'a}bor Sz{\'e}kely and Maria Rizzo.
\newblock A new test for multivariate normality.
\newblock \emph{Journal of Multivariate Analysis}, 93:\penalty0 58--80, 2005.

\bibitem[Vapnik et~al.(1997)Vapnik, Golowich, and Smola]{vapnik97support}
Vladimir Vapnik, Steven Golowich, and Alex Smola.
\newblock Support vector method for function approximation, regression
  estimation and signal processing.
\newblock In \emph{Advances in Neural Information Processing Systems
  (NeurIPS)}, pages 281--287, 1997.

\bibitem[Wheeden and Zygmund(2015)]{wheeden15measure}
Richard~L. Wheeden and Antoni Zygmund.
\newblock \emph{Measure and Integral: An Introduction to Real Analysis}.
\newblock CRC Press, 2015.

\bibitem[Zinger et~al.(1992)Zinger, Kakosyan, and
  Klebanov]{zinger92characterization}
Abram~A. Zinger, Ashot Kakosyan, and Lev Klebanov.
\newblock A characterization of distributions by mean values of statistics and
  certain probabilistic metrics.
\newblock \emph{Journal of Soviet Mathematics}, 1992.

\end{thebibliography}
